\documentclass[12pt]{amsart}
\usepackage{amsmath}
\usepackage{amsfonts}
\usepackage{amssymb}
\usepackage{amscd}
\usepackage{mathtools}
\usepackage{amsxtra}
\usepackage{amsthm}
\usepackage{graphicx}
\usepackage[abbrev,alphabetic]{amsrefs}
\RequirePackage[dvipsnames,usenames]{color}
\usepackage{soul,xcolor}
\setstcolor{red}
\usepackage{stmaryrd}
\usepackage{booktabs}
\usepackage{multirow}
\newtagform{tiny}{\tiny(}{)}
\usepackage{aliascnt}

\usepackage{mathtools}
\usepackage{hyperref}
\usepackage[margin=1.25in]{geometry}
\usepackage{mathrsfs}

\usepackage{amsthm}
\usepackage{comment}
\usepackage[all,cmtip]{xy}
\usepackage{tikz-cd}
\usetikzlibrary{cd}

\tikzcdset{
  cells={font=\everymath\expandafter{\the\everymath\displaystyle}},
}

\usepackage[all]{xy}

\usepackage{cleveref}

\makeatletter
\def\@tocline#1#2#3#4#5#6#7{\relax
  \ifnum #1>\c@tocdepth 
  \else
    \par \addpenalty\@secpenalty\addvspace{#2}%
    \begingroup \hyphenpenalty\@M
    \@ifempty{#4}{%
      \@tempdima\csname r@tocindent\number#1\endcsname\relax
    }{%
      \@tempdima#4\relax
    }%
    \parindent\z@ \leftskip#3\relax \advance\leftskip\@tempdima\relax
    \rightskip\@pnumwidth plus4em \parfillskip-\@pnumwidth
    #5\leavevmode\hskip-\@tempdima
      \ifcase #1
       \or\or \hskip 1em \or \hskip 2em \else \hskip 3em \fi%
      #6\nobreak\relax
    \hfill\hbox to\@pnumwidth{\@tocpagenum{#7}}\par
    \nobreak
    \endgroup
  \fi}
\makeatother

\renewcommand{\P}{\mathbb{P}}
\newcommand{\Z}{\mathbb{Z}}
\newcommand{\Q}{\mathbb{Q}}

\newcommand{\F}{\mathbb{F}}

\newcommand{\cO}{\mathcal{O}}

\newcommand{\m}{\mathfrak{m}}

\newcommand{\perf}{\mathrm{perf}}
\newcommand{\Frac}{\mathrm{Frac}}
\newcommand{\Proj}{\mathrm{Proj}}
\newcommand{\wt}{\widetilde}

\makeatletter
\DeclareRobustCommand{\graded}{%
    \@ifnextchar\bgroup{\graded@with}{\ensuremath{\mathrm{gr}}}%
}
\newcommand{\graded@with}[1]{\ensuremath{#1\mathrm{-gr}}}
\makeatother

\DeclareMathOperator*{\grlim}{grlim}
\DeclareMathOperator*{\grcolim}{grcolim}

\newcommand{\Lgrotimes}{\mathbin{\otimes}^{L\graded}}

\newcommand{\cLgrotimes}{\mathbin{\widehat{\otimes}}^{L\graded}}

\DeclareMathOperator{\tperfd}{\widehat{\perfd}}
\DeclareMathOperator{\grpfd}{\mathrm{grpfd}}
\DeclareMathOperator{\tgrpfd}{\widehat{\grpfd}}
\DeclareMathOperator{\pfd}{pfd}

\makeatletter
\DeclareRobustCommand{\comp}[1]{%
    \@ifnextchar\bgroup{\comp@with{#1}}{\comp@without{#1}}%
}
\newcommand{\comp@with}[2]{\Lambda_{#1}(#2)}
\newcommand{\comp@without}[1]{\ensuremath{#1\mathrm{-comp}}}
\makeatother
\newcommand{\dcomp}[2]{d\Lambda_{#1}(#2)}
\newcommand{\grcomp}[2]{\Lambda^{\graded}_{#1}(#2)}
\newcommand{\dgrcomp}[2]{d\Lambda^{\graded}_{#1}(#2)}
\DeclareMathOperator{\grwedge}{{\wedge \graded}}

\def\var{\overline}

\DeclareMathOperator{\Supp}{Supp}
\DeclareMathOperator{\Spec}{Spec}

\DeclareMathOperator{\Hom}{Hom}

\DeclareMathOperator{\Pic}{Pic}

\newcommand{\hProj}{\widehat{\Proj}}

\theoremstyle{plain}
\newtheorem{theorem}{Theorem}[section]

\newtheorem{theoremA}{Theorem}

\crefname{theoremA}{Theorem}{Theorems}
\Crefname{theoremA}{Theorem}{Theorems}

\newaliascnt{lemma}{theorem}
\newtheorem{lemma}[lemma]{Lemma}
\aliascntresetthe{lemma}
\crefname{lemma}{Lemma}{Lemmas}
\Crefname{lemma}{Lemma}{Lemmas}

\newaliascnt{proposition}{theorem}
\newtheorem{proposition}[proposition]{Proposition}
\aliascntresetthe{proposition}
\crefname{proposition}{Proposition}{Propositions}
\Crefname{proposition}{Proposition}{Propositions}

\newaliascnt{corollary}{theorem}
\newtheorem{corollary}[corollary]{Corollary}
\aliascntresetthe{corollary}
\crefname{corollary}{Corollary}{Corollaries}
\Crefname{corollary}{Corollary}{Corollaries}

\newaliascnt{claim}{theorem}

\aliascntresetthe{claim}
\crefname{claim}{Claim}{Claims}
\Crefname{claim}{Claim}{Claims}

 \newtheorem*{claim*}{Claim}

\theoremstyle{definition}
\newaliascnt{definition}{theorem}
\newtheorem{definition}[definition]{Definition}
\aliascntresetthe{definition}
\crefname{definition}{Definition}{Definitions}
\Crefname{definition}{Definition}{Definitions}

\newaliascnt{notation}{theorem}
\newtheorem{notation}[notation]{Notation}
\aliascntresetthe{notation}
\crefname{notation}{Notation}{Notations}
\Crefname{notation}{Notation}{Notations}

\newaliascnt{example}{theorem}
\newtheorem{example}[example]{Example}
\aliascntresetthe{example}
\crefname{example}{Example}{Examples}
\Crefname{example}{Example}{Examples}

\theoremstyle{remark}
\newaliascnt{remark}{theorem}
\newtheorem{remark}[remark]{Remark}
\aliascntresetthe{remark}
\crefname{remark}{Remark}{Remarks}
\Crefname{remark}{Remark}{Remarks}

\newaliascnt{setting}{theorem}

\aliascntresetthe{setting}
\crefname{setting}{Setting}{Settings}
\Crefname{setting}{Setting}{Settings}

\newaliascnt{construction}{theorem}
\newtheorem{construction}[construction]{Construction}
\aliascntresetthe{construction}
\crefname{construction}{Construction}{Constructions}
\Crefname{construction}{Construction}{Constructions}

\numberwithin{equation}{section}



\usepackage{mymacros_common_Ryo, mymacros_english_paper_Ryo}

\title{A local-global correspondence for perfectoid purity}
\author{Ryo Ishizuka}
\address{Institute of Science Tokyo, Tokyo 152-8551, Japan}
\email{ishizuka.r.ac@m.titech.ac.jp}

\author{Shou Yoshikawa}
\address{Institute of Science Tokyo, Tokyo 152-8551, Japan}
\email{yoshikawa.s.9fe9@m.isct.ac.jp}



\begin{document}

\begin{abstract}
    We introduce (lim-)perfectoid splitting, which is a global variant of (lim-)perfectoid purity.
    Our main result establishes a correspondence between the lim-perfectoid splitting of projective schemes and the lim-perfectoid purity of their Gorenstein section rings.
    As an application, we construct a new supply of examples of lim-perfectoid pure rings that go beyond the previously known complete intersection or splinter-type cases.
\end{abstract}

\maketitle 

\tableofcontents

\section{Introduction}

This paper studies a local--global correspondence for singularities in mixed-characteristic algebraic geometry through perfectoid theory introduced by Scholze \cite{scholze2012Perfectoida}. In characteristic \(p\), the prototype is provided by Frobenius splitting and related \(F\)-singularities: global Frobenius-theoretic properties of projective varieties are reflected in the \(F\)-singularities of their section rings \cites{Watanabe1991,SchwedeSmith2010}. 
In particular, this correspondence makes the \(F\)-splitting computable in many examples, since one can apply local criteria such as Fedder's criterion \cite{Fed} to the section ring. Our goal is to develop a mixed-characteristic analogue of this picture.

In mixed characteristic, the role of Frobenius is played by maps to perfectoid rings. Indeed, a complete Noetherian local ring \(R\) of mixed characteristic is regular if and only if there exists a faithfully flat map \(R \to P\) to a perfectoid ring \(P\) \cite{bhatt2019Regular}. Motivated by this, Bhatt et al.\ introduced \emph{(lim-)perfectoid pure}\footnote{A \(p\)-adically complete Noetherian ring \(S\) is \emph{lim-perfectoid pure} (resp., \emph{perfectoid pure}) if the canonical morphism \(S \to S_{\perfd}\) to the absolute perfectoidization of \(S\) is pure (resp., there exists a pure morphism \(S \to P\) to a perfectoid ring \(P\)).} singularities as mixed-characteristic analogues of \(F\)-pure singularities \cite{bhatt2024Perfectoid}.


From the perspective of lim-perfectoid pure singularities, it is natural to consider a global analogue obtained by replacing the local splitting condition with its geometric counterpart. 
In order to explain global theory, we prepare notations.

\begin{notation} \label{NotationIntroduction}
    Let \(V\) be a \(p\)-torsion-free discrete valuation ring with residue characteristic \(p\), and let \(X\) be a flat projective scheme over \(V\) such that \(H^0(X, \mcalO_X) \cong V\). 
    Write $d:=\dim X$.
    Let \(\widehat{X}\) be the \(p\)-adic completion of \(X\).
    Let \(\mcalO_X(1)\) be an ample line bundle on \(X\) and set the section ring \(R \defeq \bigoplus_{m \geq 0} H^0(X, \mcalO_X(m))\).
\end{notation}

A natural global analogue of perfectoid purity is provided by the notion of \emph{perfectoid split}. Namely, we say that \(X\) is \emph{perfectoid split} if there exists an affine morphism \(\pi \colon Y \to X\) such that the \(p\)-adic formal completion \(\widehat{Y}\) is a perfectoid formal scheme and the natural map \(\mathcal O_X \to \pi_*\mathcal O_Y\) is ind-split, that is, a filtered colimit of split morphisms in \(\QCoh(X)\) (\Cref{DefDPS}).

In one direction, the relation with the local theory is relatively well behaved: if the section ring $R^{\wedge p}$ is perfectoid pure, then $X$ is perfectoid split (\Cref{LocalGlobalPure}) by the theory of graded perfectoid rings \cite{ishizuka2025Graded}. 
The converse direction is more delicate, since it is not clear in general whether a perfectoid cover of \(X\) gives rise to a corresponding perfectoid cover of \(R\).
For this reason, to establish a global-to-local correspondence, we instead adopt a different approach based on absolute perfectoidization, which leads us to the notion of \emph{lim-perfectoid splitting}.

\begin{definition}[{\Cref{DefDPS}}, cf.~\cites{bhatt2025Aspects}] \label{IntroductionDefGLPP}
    Let \(\mcalO_{\widehat{X}, \perfd}\) be the absolute perfectoidization of the structure sheaf \(\mcalO_{\widehat{X}}\) of \(\widehat{X}\) defined in \cite{ishizuka2026Algebraization}.
    We say that \(X\) is \emph{lim-perfectoid split} if the canonical morphism \(\mcalO_{\widehat{X}} \to \mcalO_{\widehat{X}, \perfd}\) in \(\mcalD_{\qcoh}(\widehat{X})\) is ind-split.
\end{definition}

Recently, Bhatt proved that lim-perfectoid split regular schemes satisfy the Kodaira-type vanishing theorem (\cite{bhatt2025Aspects}*{Proposition 11.2.16}; see also \Cref{RemarkKodairaVanishing}). As in the case of \(F\)-splitting in positive characteristic (cf.~\cite{mehta1985Frobenius}), this shows that lim-perfectoid splitting is a useful notion for proving vanishing theorems in mixed-characteristic. 
For this reason, it is important to have effective criteria for detecting when a scheme is lim-perfectoid split.

Our main result compares perfectoid singularities of the section ring $R$ and (lim-)perfectoid splitting of $X$:

\begin{theoremA}[{\Cref{SectionLocalGlobal}}] \label{MainTheoremLocalGlobal}
    Keep the setting of \Cref{NotationIntroduction}.
    Assume that \(X\) is Cohen--Macaulay and either \(\mcalO_X \cong \omega_X\), \(\mcalO_X(1) = \omega_X^{-1}\), or \(\cO_X(1)=\omega_X\).
    Then the following assertions hold:
    \begin{enumerate}
        \item If $H^{d+1}_{(p,R_+)}(R) \to H^{d+1}_{(p,R_+)}(R_{\perfd})$ is injective, then \(X\) is lim-perfectoid split (\Cref{local-to-global-PS}).
        \item In either the case of \(\dim X \geq 3\) and \(\mcalO_X \cong \omega_X\), or \(\cO_X(1)=\omega_X^{-1}\), the converse direction of (1) holds (\Cref{Global-local-CY} and \Cref{Global-local-Fano}).
        \item Under the assumption on (2), we further assume that $X$ is complete intersection on a projective space.
        Then $X$ is perfectoid split if and only if $X$ is lim-perfectoid split (\Cref{CICase}).
    \end{enumerate}
\end{theoremA}

\noindent
Remark that, under the assumption on \Cref{MainTheoremLocalGlobal}, if the section ring \(R\) is Cohen--Macaulay, then the injectivity in \Cref{MainTheoremLocalGlobal}~(1) is equivalent to the lim-perfectoid purity of \(R^{\wedge p}\)(\Cref{rmk:Gorensteiness-section-ring} and \Cref{EquivGradedPerfdPure}).

A key ingredient is a fiber sequence which relates the local cohomology of the absolute graded perfectoidization \(R_{\grpfd}\) defined in \cite{ishizuka2026Algebraization} to the cohomology of the absolute perfectoidization \(\mcalO_{\widehat{X}, \perfd}\):

\begin{theoremA}[{\Cref{GradedAbsPerfdSheafCoho}}] \label{MainTheoremFiberSequence}
    Keep the setting in \Cref{NotationIntroduction}.
    Then we can equip both \(R\Gamma_{(p, R_+)}(R_{\perfd})\) and \(R\Gamma_{(p)}(R_{\perfd})\) with a \(\setZ[1/p]\)-grading, and there exists a fiber sequence
    \begin{equation*}
        R\Gamma_{(p, R_+)}(R_{\perfd}) \to R\Gamma_{(p)}(R_{\perfd}) \to \bigoplus_{q \in \setZ[1/p]} R\Gamma_{(p)}(R\Gamma(\widehat{X}, \mcalO_{\widehat{X}, \perfd}(q)))
    \end{equation*}
    in the derived (\(\infty\)-)category of \(\setZ[1/p]\)-graded \(R\)-modules, where \(\mcalO_{\widehat{X}, \perfd}(q)\) is the twist by \(q \in \setZ[1/p]\) (\Cref{RationalTwistAbsPerfd}).
\end{theoremA}

One serious difficulty in proving \Cref{MainTheoremLocalGlobal} is that the middle term in the fiber sequence in \Cref{MainTheoremFiberSequence} is itself a derived object. 
In contrast to the local-global correspondence in positive characteristic, where the corresponding middle term is a discrete ring, here one does not have the vanishing on \(H^{d}_{(p)}(R_{\perfd})\).
This is one of the main reasons why our argument is more delicate than in positive characteristic. 
This lack of control is also reflected in the assumptions of \Cref{MainTheoremLocalGlobal}(2) (see \Cref{GradedZeroPartInjective} for details). 

The comparison theorem \Cref{MainTheoremLocalGlobal} makes it possible to pass between the local theory of the section ring $R$ and the global geometry of \(X\), and hence to take advantage of both viewpoints. 
In particular, by combining \Cref{MainTheoremLocalGlobal} with the Fedder-type criteria for perfectoid purity established in \cites{yoshikawa2025Computation,yoshikawa2025Criterion}, we obtain a practical method for constructing examples. These examples further show that perfectoid splitting and lim-perfectoid splitting are not merely mixed-characteristic analogues of \(F\)-splitting, but exhibit new phenomena specific to mixed characteristic.

\begin{example} \label{IntroductionExample}
Let $k$ be an algebraically closed field of characteristic $p$.
\begin{enumerate}
\item Any \(W(k)\)-lift of an elliptic curve over \(k\) is perfectoid split (\Cref{elliptic-curve}).
\item The canonical \(W(k)\)-lift \(X\) of an ordinary abelian variety over \(k\) is perfectoid split (\Cref{ExampleWeaklyOrdinary}).
\item Let $\var{X}$ be a hypersurface Calabi--Yau variety over $k$ such that the Artin--Mazur height of $X$ is infinite and $\dim \var{X} \geq 2$.
If $p=2$, then there exist  $W(k)$-lifts $X_1$ and $X_2$ of $\var{X}$ such that $X_1$ is not lim-perfectoid split and $X_2$ is perfectoid split (\Cref{CY}(2)).
\item Let $X \subseteq \P^3_{W(k)}$ be a smooth hypersurface of degree $4$ over $W(k)$.
If $X$ is lim-perfectoid split, then the $p$-adic completion of every section ring of $X$ is lim-perfectoid pure (\Cref{every-section-ring}).
Furthermore, if $p \neq 2$, then $X$ is automatically lim-perfectoid split by \cite{Yoshikawa26}*{Thoerem~A}.
This yields a new supply of examples of mixed-characteristic singularities that go beyond the previously known complete intersection or splinter-type cases (cf.~\Cref{non-lci}).
\end{enumerate}
\end{example}

\subsection*{Structure of this paper}
In \Cref{SectionPreliminaries}, we introduce notation and terminology and review our previous work. In \Cref{SectionPerfectoidSplitting}, we introduce several notions of perfectoid purity for (formal) schemes. In \Cref{SectionLocalCoho}, we develop the basic theory of local cohomology for derived graded modules and construct a fiber sequence (\Cref{MainTheoremFiberSequence}), which is a key ingredient in the proof of \Cref{MainTheoremLocalGlobal}. In \Cref{SectionLocalGlobal}, we prove \Cref{MainTheoremLocalGlobal} using the fiber sequence.

\subsection*{Acknowledgments}
The authors would like to thank Jack J. Garzella, L\'eo Navarro Chafloque, and Joe Waldron for helpful discussions.
This work was started while we attended the conference `\(p\)-adic and Characteristic \(p\) Methods in Algebraic Geometry' at EPFL and `the Summer Research Institute in Algebraic Geometry' at Colorado State University in 2025. We are very grateful for these opportunities and their hospitality.

The first-named author was supported by JSPS KAKENHI Grant number 24KJ1085.
The second-named author was supported by JSPS KAKENHI Grant number JP24K16889.

\section{Preliminaries} \label{SectionPreliminaries}

\subsection{Notation and terminology}

Throughout this paper, we use the following notation and terminology.

\subsubsection{General notation}
\begin{enumerate}
    \item We fix a prime number \(p\).
    \item We only consider commutative rings with unity, which we simply refer to as \emph{rings}. The category of (commutative) rings is denoted by \(\CRing\) (not \(\CAlg\)).
\end{enumerate}

\subsubsection{Higher categorical stuff}
\begin{enumalphp}
    \item We use the notion of an \(\infty\)-category (more precisely, an \((\infty,1)\)-category). Our main references are \cites{lurie2017Higher, lurie2018Spectral}.
    \item The \(\infty\)-category of \(\setE_{\infty}\)-ring spectra is denoted by \(\CAlg\); this category admits all small limits and colimits. We simply call an object of \(\CAlg\) an \emph{\(\setE_{\infty}\)-ring}. If we take the slice category over an \(\setE_{\infty}\)-ring \(R\), we denote the \(\infty\)-category of \(\setE_{\infty}\)-\(R\)-algebras as \(\CAlg_R\).
    \item If we say \emph{discrete} rings and modules, we mean commutative rings and modules without higher homotopy groups. Since we rarely use topological rings equipped with the discrete topology (and we make this explicit when doing so), we hope that no confusion will arise.
\end{enumalphp}

\subsubsection{Graded rings}
\begin{enumalphp}
    \item Let \(G\) be a torsion-free abelian group with identity element \(0\).
    \item A \emph{\(G\)-graded ring} is a pair consisting of a commutative ring \(R\) and a family of additive subgroups \(\{R_g\}_{g \in G}\) such that \(R\) admits a decomposition \(R = \bigoplus_{g \in G} R_g\) and satisfies \(R_g R_{g'} \subseteq R_{g + g'}\) for any \(g, g' \in G\). A \emph{homogeneous element} of \(R\) is an element of \(R_g\) for some \(g \in G\), and a \emph{homogeneous ideal} is an ideal of \(R\) generated by homogeneous elements.
    The subset $\{g \in G \mid R_g \neq 0\}$ of $G$ is called the \emph{support of $R$} and is denoted by $\Supp(R)$.
    Furthermore, for a submonoid $H$ of $G$, if the support of $R$ is contained in $H$, then we say that $R$ is an \emph{$H$-graded ring}.
    \item A \emph{\(G\)-graded ring morphism} between \(G\)-graded rings \(R = \bigoplus_{g \in G} R_g\) and \(S = \bigoplus_{g \in G} S_g\) is a ring homomorphism \(\varphi \colon R \to S\) such that \(\varphi(R_g) \subseteq S_g\) for any \(g \in G\).
    \item For a (discrete) \(G\)-graded ring \(R\), we will use the \(\infty\)-category \(\mcalD_{\graded{G}}(R)\) of graded \(R\)-modules defined in \cite{ishizuka2026Derived}*{Section 3}.
    \item In \(\mcalD_{\graded{G}}(R)\) of a \(G\)-graded ring \(R\), the (co)limits and symmetric monoidal products are denoted by \(\grlim\), \(\grcolim\), and \(- \Lgrotimes -\) (if they exist), whereas in \(\mcalD(R)\) of a discrete ring \(R\), i.e., derived limits and derived tensor products, they are often denoted by \(\lim\) (or \(R\lim\)), \(\colim\), and \(- \otimes^L -\).
    \item A \emph{\(G\)-graded adic ring} is a \(G\)-graded ring \(R\) with an adic topology for a finitely generated homogeneous ideal \(I\). We say that such ideal \(I\) is a \emph{homogeneous ideal of definition} of \(R\).
    \item A \(G\)-graded adic ring \(R\) is said to be \emph{gradedwise complete} if the canonical morphism \(R \to \bigoplus_{g \in G} \lim_{n \geq 1}(R/I^n)_g\) is an isomorphism of \(G\)-graded rings. See more comprehensively \cites{ishizuka2025Graded, ishizuka2026Derived}.
\end{enumalphp}

\subsubsection{Topological rings}
\begin{enumalphp}
    \item A topological ring is said to have a linear topology if there exists a fundamental system of neighborhoods of \(0\) consisting of ideals. A linear topological ring \(R\) is called \emph{complete} if every Cauchy net in \(R\) converges uniquely. If there exists the universal complete topological \(R\)-algebra, then it is called the \emph{completion} of \(R\).
    \item If a topological ring \(R\) has an ideal \(I\) such that \(I\) is finitely generated and the set \(\{I^n\}_{n \geq 0}\) forms a fundamental system of neighborhoods of \(0\), then we call \(R\) an \emph{adic ring} and \(I\) an \emph{ideal of definition} of \(R\).
    \item Morphisms between adic rings are assumed to be continuous ring homomorphisms.
\end{enumalphp}

\subsubsection{\texorpdfstring{\(p\)-adic stuff}{p-adic stuff}}
\begin{enumalphp}
    \item We freely use the notion of perfectoid rings, prisms, and related concepts, e.g., perfectoidization, following \cites{bhatt2018Integral, bhatt2022Prismsa}.
    \item An adic ring \(R\) is called an \emph{adic perfectoid ring} if \(p\) is a topologically nilpotent element of \(R\) and the underlying ring \(R\) is a perfectoid ring (\cite{takaya2025Relative}*{Definition 4.4}).
    \item For a base adic ring \(A\) with an ideal of definition \(I\), we denote by \(\Perfd_A^{\wedge I}\) the category of \(I\)-adically complete perfectoid \(A\)-algebras.
\end{enumalphp}

\subsubsection{Formal schemes}
\begin{enumalphp}
    \item A \emph{formal scheme} \(\mscrX\) is a topologically ringed space \(\mscrX = (\abs{\mscrX}, \mcalO_{\mscrX})\) that is Zariski locally isomorphic to the affine formal spectrum \(\Spf(R)\) of a linear topological ring \(R\). Our definition follows \cite{grothendieck1971Elements}*{Ch.I \S 10}.
    \item A formal scheme \(\mscrX\) is said to be \emph{adic}, \emph{quasi-compact}, or \emph{separated} if it is Zariski locally isomorphic to the affine formal spectrum \(\Spf(R)\) of an adic ring \(R\), if \(\mscrX\) is quasi-compact as a topological space, or if any intersection of two affine open subsets of \(\mscrX\) is affine, respectively.
    \item For an adic quasi-compact formal scheme \(\mscrX\), a finitely generated sheaf of ideals \(\mcalI \subseteq \mcalO_{\mscrX}\) is called an \emph{ideal of definition} of \(\mscrX\) if \(\{\mcalI(\mscrU)^n\}_{n \geq 0}\) forms a fundamental system of neighborhoods of \(0\) in \(\mcalO_{\mscrX}(\mscrU)\) for any affine open subset \(\mscrU\).
    \item Given a scheme \(X\) and a closed subscheme \(Z\) of \(X\), we can construct the \emph{formal completion \(\widehat{X}\) of \(X\) along \(Z\)}, which is a formal scheme whose underlying topological space is \(Z\) and whose structure sheaf is the restriction to \(Z\) of the sheaf \(\lim_{n \geq 1} \mcalO_X/\mcalI_Z^n\) on \(X\), where \(\mcalI_Z\) is the quasi-coherent sheaf of ideals of \(\mcalO_X\) defining \(Z\).
    \item For an adic separated formal scheme \(\mscrX\) with an ideal of definition \(\mcalI\), we will use the \emph{derived \(\infty\)-category of quasi-coherent \(\mcalO_{\mscrX}\)-modules} \(\mcalD_{\qcoh}(\mscrX) \defeq \lim_{n \geq 1} \mcalD_{\qcoh}(\mscrX_n)\), where \(\mscrX_n\) is the closed subscheme defiend by \(\mcalI^n\). See \cite{ishizuka2026Algebraization}*{Appendix C} for details.
\end{enumalphp}

\subsubsection{Completion functors}

Let $R$ be a ring and let $I$ be a finitely generated ideal of $R$.
Let $S$ be a $G$-graded ring and let $J$ be a finitely generated homogeneous ideal of $S$.
\begin{enumalphp}
    \item The symbol \(\comp{I}{-}\) is used only for the (classical) \(I\)-adic completion of discrete rings and modules.
    \item The symbol \(\dcomp{I}{-}\) is used to denote the derived (\(I\)-)completion of objects in a derived category such as \(\mcalD(R)\).
    \item The symbol \(\grcomp{J}{-}\) denotes the \emph{gradedwise (\(J\)-)completion} of discrete graded rings and modules as defined in \cite{ishizuka2025Graded}*{Construction 3.3}. In \emph{loc. cit.}, it is denoted by \((-)^{\grwedge}\), whereas in this paper, we employ the symbol \(\grcomp{J}{-}\) to avoid confusion with other derived completion functors.
    \item The symbol \(\dgrcomp{J}{-}\) is the \emph{derived gradedwise \(J\)-completion} of objects in \(\mcalD_{\graded{G}}(S)\), which is defined in \cite{ishizuka2026Derived}*{Definition 3.5}.
    \item The full subcategory of \emph{derived \(I\)-complete} objects in \(\mcalD(R)\) is denoted by \(\mcalD^{\comp{I}}(R)\) and the full subcategory of \emph{derived gradedwise \(J\)-complete} objects in \(\mcalD_{\graded{G}}(S)\) is denoted by \(\mcalD_{\graded{G}}^{\comp{J}}(S)\) (\cite{ishizuka2026Derived}*{Definition 4.6}).
    \item We will denote the derived (gradedwise) complete tensor products by \(- \widehat{\otimes}^L -\) and \(- \cLgrotimes -\) in \(\mcalD^{\comp{I}}(R)\) and \(\mcalD_{\graded{G}}^{\comp{J}}(S)\), respectively.
    \item We will also use the symbols \(\widehat{(-)}\) and \((-)^{\wedge}\) when convenient; their meanings will be made clear from the context.
\end{enumalphp}

\subsection{Perfectoid rings and formal schemes}

First, we recall the definition and properties of adic perfectoid rings.

\begin{definition}[{\cite{takaya2025Relative}*{Definition 4.4}}] 
    An adic ring \(R\) is called an \emph{adic perfectoid ring} if \(p\) is a topologically nilpotent element of \(R\) and the underlying ring \(R\) is a perfectoid ring.
    If an adic perfectoid ring \(R\) is complete with respect to the adic topology on \(R\), then we say that \(R\) is a \emph{complete adic perfectoid ring}.
\end{definition}

We also introduce the corresponding notion for formal schemes.

\begin{definition}[{\cite{takaya2025Relative}*{Definition 4.25}}] \label{DefPerfectoidFormalScheme}
    An adic formal scheme \(\mscrX\) on which \(p\) is topologically nilpotent is called a \emph{perfectoid formal scheme} if there exists an affine open covering \(\{\mscrU_i\}_{i \in I}\) of \(\mscrX\) such that for each \(i \in I\), the adic ring \(\mcalO_{\mscrX}(\mscrU_i)\) is a complete adic perfectoid ring.
\end{definition}

Nontrivially, any affine section of a perfectoid formal scheme is an adic perfectoid ring:

\begin{proposition}[{\cite{takaya2025Relative}*{Proposition 4.26}}] \label{PropPerfectoidFormalScheme}
    An adic formal scheme \(\mscrX\) is a perfectoid formal scheme if and only if the ring of sections \(\mcalO_{\mscrX}(\mscrU)\) for any affine open subset \(\mscrU\) is a complete adic perfectoid ring.
\end{proposition}

We will also use the following properties of derived completions for formal schemes.

\begin{lemma}[{Special case of \cite{ishizuka2026Algebraization}*{Lemma C.12}}] \label{DerivedCompletionFormalSchemes}
    Let \(X\) be a locally Noetherian quasi-compact separated scheme over an adic Noetherian ring \(R\) with an ideal of definition \(I\), and let \(\mscrX\) be the \(I\)-adic formal completion of \(X\).
    Then for any \(\mcalF \in \mcalD_{\qcoh}(X)\), we have a canonical isomorphism
    \begin{equation*}
        \dcomp{I}{R\Gamma(X, \mcalF)} \xrightarrow{\cong} R\Gamma(\mscrX, \mcalF^{\wedge})
    \end{equation*}
    in \(\mcalD(R)\), where \(\mcalF^{\wedge}\) is the limit \(\lim_{n \geq 1} (\mcalF \otimes^L_{\mcalO_X} \mcalO_{X_n})\) in \(\mcalD_{\qcoh}(\mscrX)\) with \(X_n\) being the closed subscheme of \(X\) defined by \(I^n\).
\end{lemma}

\begin{lemma}[{cf.~\cite{ishizuka2026Algebraization}*{Lemma C.13}}] \label{SymMonDerivedCompletionFormalSchemes}
    Keep the setting of \Cref{DerivedCompletionFormalSchemes}.
    Then the construction \(\mcalF \mapsto \mcalF^{\wedge}\) in \Cref{DerivedCompletionFormalSchemes} defines a functor
    \begin{equation*}
        (-)^{\wedge} \colon \mcalD_{\qcoh}(X) \to \mcalD_{\qcoh}(\mscrX); \quad \mcalF \mapsto \lim_{n \geq 1} (Lj_n^*\mcalF)
    \end{equation*}
    (where \(j_n \colon X_n \to X\) denotes the natural closed immersion) that preserves colimits. Furthermore, for any \(\mcalF, \mcalG \in \mcalD_{\qcoh}(X)\), the canonical morphism
    \begin{equation*}
        (\mcalF \otimes^L_{\mcalO_X} \mcalG)^{\wedge} \xrightarrow{\cong} \dcomp{I}{\mcalF^{\wedge} \otimes^L_{\mcalO_{\mscrX}} \mcalG^{\wedge}}
    \end{equation*}
    is an isomorphism in \(\mcalD(\mscrX_{\Zar}, \mcalO_{\mscrX})\), where \(\dcomp{I}{-}\) is the derived \(I\mcalO_{\mscrX}\)-completion.
\end{lemma}

\subsection{Graded perfectoid rings}

Following \cite{ishizuka2025Graded}, we introduce the notion of graded perfectoid rings.
We recall the definitions and properties that will be used later. See \cite{ishizuka2025Graded} for more details.

\begin{definition}[{cf.~\cite{ishizuka2025Graded}*{Definition 4.1}}] \label{DefGradedPerfd}
    We define a graded variant of perfectoid rings as follows.
    \begin{enumerate}
        \item Let $R$ be a $G$-graded ring.
        We say that $R$ is \emph{a $G$-graded perfectoid ring} if $R$ is gradedwise $p$-adic complete and $R^{\wedge p}$ is a perfectoid ring.
        \item Let \(R = (R, R_{\graded})\) be a pro-\(G\)-graded ring. We say that \(R\) is \emph{a pro-\(G\)-graded perfectoid ring} if \(p\) is a topologically nilpotent element of \(R_{\graded}\) and \(R_{\graded}\) is a \(G\)-graded perfectoid ring.
    \end{enumerate}
    On a base \(G\)-graded adic ring \(A\) with a homogeneous ideal of definition \(I\), the category of \(I\)-adic \(G\)-graded perfectoid \(A\)-algebras is written by \(\Perfd_{\graded{G}}^{\wedge I}(A)\).
\end{definition}

In our previous paper \cite{ishizuka2026Derived}, we have introduced the derived gradedwise completion, so we need to compare it with the usual gradedwise completion for graded perfectoid rings as follows.

\begin{lemma}[{\cite{ishizuka2026Algebraization}*{Lemma 2.16}}] \label{DerivedGrcompGradedPerfd}
    Let \((R, R_{\graded})\) be a \(p\)-adic pro-\(G\)-graded perfectoid ring.
    Let $I$ be a finitely generated homogeneous ideal of $R_{\graded}$ containing $p$.
    Then the canonical morphism
    \begin{equation*}
        \dgrcomp{I}{R_{\graded}} \to \grcomp{I}{R_{\graded}}
    \end{equation*}
    is an isomorphism in \(\mcalD_{\graded{G}}(R_{\graded})\).
\end{lemma}


\begin{lemma}[{\cite{ishizuka2026Algebraization}*{Lemma 2.17}}] \label{LocalizationGradedPerfd}
Let \((R, R_{\graded})\) be an \(I\)-adic pro-\(G\)-graded perfectoid ring with a homogeneous ideal of definition \(I\) containing \(p\), and let \(f\) be a homogeneous element of \(R_{\graded}\).
\begin{enumerate}
    \item The pair \((\comp{I}{R[1/f]}, \grcomp{I}{R_{\graded}[1/f]})\) is an \(I\)-adic pro-\(G\)-graded perfectoid ring.
    \item The canonical morphism
    \[
      \dgrcomp{I}{R_{\graded}[1/f]} \;\longrightarrow\; \grcomp{I}{R_{\graded}[1/f]}
    \]
    in \(\mcalD_{\graded{G}}(R_{\graded})\) is an isomorphism.
    \item The canonical morphism
    \begin{equation*}
        \dcomp{I}{R_{\graded}[1/f]} \to \comp{I}{R_{\graded}[1/f]}
    \end{equation*}
    in \(\mcalD(R_{\graded})\) is an isomorphism.
\end{enumerate}
\end{lemma}

On the category of graded perfectoid algebras, we have the following cofinality results.

\begin{lemma}[{\cite{ishizuka2026Algebraization}*{Lemma 3.9}}] \label{GradedPerfdCofinalFunctors}
    Let \(R\) be a \(G\)-graded adic ring with \(G = G[1/p]\) and let \(I\) be a homogeneous ideal of definition of \(R\) containing \(p\).
    Let \(H\) be the submonoid of \(G\) generated by the support \(\Supp(R)\) of \(R\).
    Then the canonical functors of categories
    \begin{equation*}
        \Perfd_{\graded{H[1/p]}}^{\wedge I}(R) \hookrightarrow \Perfd_{\graded{G}}^{\wedge I}(R) \hookrightarrow \Perfd_{\graded{G}}^{\wedge}(R)
    \end{equation*}
    are cofinal, where \(\Perfd_{\graded{G}}^{\wedge}(R)\) (resp., \(\Perfd_{\graded{G}}^{\wedge I}(R)\)) is the category of gradedwise complete \(G\)-graded adic (resp., \(I\)-adic) perfectoid \(R\)-algebras and \(\Perfd_{\graded{H[1/p]}}^{\wedge I}(R)\) is the full subcategory of \(\Perfd_{\graded{G}}^{\wedge I}(R)\) consisting of all gradedwise complete \(G\)-graded \(I\)-adic perfectoid \(R\)-algebras \(P\) such that \(\Supp(P) = H[1/p]\).
\end{lemma}

\begin{proposition}[{\cite{ishizuka2026Algebraization}*{Proposition 3.7}}] \label{CofinalGradedPerfdForGradedRing}
    Let \(R\) be a \(G\)-graded adic ring with \(G = G[1/p]\).
    Let $H \subseteq G$ be a submonoid such that $Supp(R) \subseteq H$.
    Take any complete adic perfectoid \(R\)-algebra \(P\) with an ideal of definition \(J\) containing \(p\).
    Then there exists a gradedwise complete \(G\)-graded adic perfectoid \(R\)-algebra \(Q\) and an \(R\)-algebra morphism \(Q \to P\) such that $\Supp(Q) \subseteq H[1/p]$.
\end{proposition}

\subsection{Absolute graded perfectoidization}

In this subsection, we recall the definition and properties of the absolute graded perfectoidization defined in \cite{ishizuka2026Algebraization}.
Before that, we introduce the absolute perfectoidization for adic rings.

\begin{definition}[{\cite{ishizuka2026Algebraization}*{Definition 3.1 and Proposition 3.3}}] \label{CompletionAbsolutePerfectoidization}
    Let \(R\) be an adic topological ring in which \(p\) is topologically nilpotent, and let \(I\) be an ideal of definition of \(R\) containing \(p\).
    Then the \emph{topological absolute perfectoidization} \(R_{\tperfd}\) of \(R\) is defined as the limit
    \begin{equation*}
        R_{\tperfd} \defeq \lim_{P \in \Perfd^{\wedge I}(R)} P
    \end{equation*}
    in \(\CAlg(\mcalD(R))\), where \(\Perfd^{\wedge I}(R)\) is the category of \(I\)-adic perfectoid \(R\)-algebras.
    This is isomorphic to the derived \(I\)-completion of the absolute perfectoidization \(R_{\perfd}\) of \(R\) in the sense of \cite{bhatt2024Perfectoid}*{Definition 3.10}.
\end{definition}


\begin{definition}[{\cite{ishizuka2026Algebraization}*{Definition 3.11}}] \label{DefGradedAbsPerfd}
    Let \(R\) be a \(G\)-graded adic ring such that \(G = G[1/p]\) and let \(I\) be a homogeneous ideal of definition of \(R\) containing \(p\).
    Then the limit
    \begin{equation*}
        R_{\tgrpfd} \defeq \grlim_{P \in \Perfd_{\graded{G}}^{\wedge I}(R)} P
    \end{equation*}
    in \(\CAlg(\mcalD_{\graded{G}}(R))\) exists, where \(\Perfd_{\graded{G}}^{\wedge I}(R)\) is the category of \(I\)-adic \(G\)-graded perfectoid \(R\)-algebras.
    We call \(R_{\tgrpfd}\) the \emph{absolute graded perfectoidization} of \(R\).
\end{definition}

\begin{theorem} \label{GradedAbsPerfdProperties}
    In the situation of \Cref{DefGradedAbsPerfd}, the absolute graded perfectoidization \(R_{\tgrpfd}\) satisfies the following properties.
    \begin{enumerate}
        \item \label{GradedPerfdZeroPart} If \(R_g = 0\) or \(R_{-g} = 0\) holds for each \(g \in G\), then the degree-zero part \((R_{\tgrpfd})_0\) of \(R_{\tgrpfd}\) can be identified with the topological absolute perfectoidization \((R_0)_{\tperfd}\) of \(R_0\) (\cite{ishizuka2026Algebraization}*{Proposition 3.15}).
        \item \label{CompletionOfAbsoluteGradedPerfd} The canonical morphism \(R_{\tgrpfd} \to R_{\tperfd}\) induces an isomorphism
        \begin{equation*}
            \dcomp{I}{R_{\tgrpfd}} \xrightarrow{\cong} R_{\tperfd}
        \end{equation*}
        in \(\CAlg_R\) (\cite{ishizuka2026Algebraization}*{Corollary 3.26}).
        \item \label{RemarkGradedPerfd} The absolute graded perfectoidization \(R_{\tgrpfd}\) is derived gradedwise \(I\)-complete and its support \(\Supp(R_{\tgrpfd})\) is equal to \(H[1/p]\) for the monoid \(H\) generated by \(\Supp(R)\) (\cite{ishizuka2026Algebraization}*{Remark 3.13}).
    \end{enumerate}
\end{theorem}

\begin{corollary}[{\cite{ishizuka2026Algebraization}*{Corollary 3.31}}] \label{CommutativityLimitsPerfd}
    Keep the setting of \Cref{DefGradedAbsPerfd}.
    Let \(f\) be an homogeneous element of \(R\).
    Then there exists a canonical isomorphism
    \begin{equation*}
        \dgrcomp{I}{R_{\tgrpfd}[1/f]} \xrightarrow{\cong} (R[1/f])_{\tgrpfd}
    \end{equation*}
    in \(\CAlg(\mcalD_{\graded{G}}(R))\).
    Moreover, we have the following commutativity of limits:
    \begin{equation} \label{CommutesLimitIsom}
        \dgrcomp{I}{R_{\tgrpfd}[1/f]} = \dgrcomp{I}{(\grlim_{P \in \Perfd_{\graded{G}}^{\wedge I}(R)} P)[1/f]} \xrightarrow{\cong} \grlim_{P \in \Perfd_{\graded{G}}^{\wedge I}(R)}(\dgrcomp{I}{P[1/f]})
    \end{equation}
    in \(\CAlg(\mcalD_{\graded{G}}(R))\), where the limits run over all \(I\)-adic \(G\)-graded perfectoid \(R\)-algebras \(P\).
\end{corollary}

For a graded ring, we have the following equivalences of lim-perfectoid purity.

\begin{lemma}[{cf.~\cite{ishizuka2025Graded}}] \label{EquivGradedPerfdPure}
    Let \(R\) be a \(p\)-torsion-free \(G\)-graded Noetherian ring with the unique graded maximal ideal \(\mfrakm\) containing \(p\).
    Write \(d \defeq \dim(R_{\mfrakm})\).
    Then the following are equivalent.
    \begin{enumerate}
        \item This ring \(R\) is graded lim-perfectoid pure in the sence that the canonical morphism \(R \to R_{\grpfd}\) is pure in \(\mcalD_{\graded{G}}(R)\).
        \item The \(p\)-adic completion \(R^{\wedge p}\) is lim-perfectoid pure.
        \item The localization \(R_{\mfrakm}\) is lim-perfectoid pure.
    \end{enumerate}
    If we further assume that \(R_{\mfrakm}\) is Gorenstein, then the following condition is also equivalent.
    \begin{enumerate}
        \item[(4)] The canonical morphism \(H^d_{\mfrakm}(R) \to H^d_{\mfrakm}(R_{\perfd})\) is injective for \(d \defeq \dim(R_{\mfrakm})\).
    \end{enumerate}
\end{lemma}

\begin{proof}
    By \Cref{GradedAbsPerfdProperties} and \cite{bhatt2024Perfectoid}*{Lemma 3.21}, we have isomorphisms
    \begin{equation*}
        \dcomp{p}{R_{\grpfd}} \cong R_{\perfd} \quad \text{and} \quad \dcomp{p}{(R_{\grpfd})_{\mfrakm}} \cong (R_{\mfrakm})_{\perfd}.
    \end{equation*}
    Therefore, the equivalences follow from \Cref{PurityEquiv} below.

    Assume that \(R\) is Gorenstein.
    If \(R^{\wedge p}\) is perfectoid pure, then the morphism
    \begin{equation*}
        H^d_{\mfrakm}(R) \to H^d_{\mfrakm}(R_{\perfd})
    \end{equation*}
    is injective. Conversely, if it is injective, then \(R_{\mfrakm}\) is lim-perfectoid pure since \(R\) is Gorenstein (\cite{bhatt2024Perfectoid}*{Lemma 4.4}).
\end{proof}

First, we prove the following lemma:

\begin{lemma} \label{PurityEquiv}
    Let \(S\) be a Noetherian \(p\)-adic \(G\)-graded ring, and let \(\psi \colon S \to M\) be a morphism in \(\mcalD_{\graded{G}}(S)\).
    Then the following are equivalent:
    \begin{enumerate}
        \item This \(\psi \colon S \to M\) is pure in \(\mcalD_{\graded{G}}(S)\).
        \item The derived \(p\)-completion \(\psi^{\wedge p} \colon S^{\wedge p} \to \dcomp{p}{M}\) is pure in \(\mcalD(S^{\wedge p})\).
        \item The induced morphism \(\psi'_{\mfrakm} \colon S_{\mfrakm} \to \dcomp{p}{M_{\mfrakm}}\) is pure in \(\mcalD(S_{\mfrakm})\).
        \item The localization \(\psi_{\mfrakm} \colon S_{\mfrakm} \to M_{\mfrakm}\) is pure in \(\mcalD(S_{\mfrakm})\).
    \end{enumerate}
\end{lemma}

\begin{proof}
    (1) \(\Rightarrow\) (2): This is clear.

    (2) \(\Rightarrow\) (3): As in the proof of \cite{ishizuka2025Graded}*{Lemma 2.10}, the assumption that \(S^{\wedge p} \to \dcomp{p}{M}\) is pure in \(\mcalD(S)\) implies that the composition
    \begin{equation*}
        S_{\mfrakm} \to (S_{\mfrakm})^{\wedge p} \cong ((S^{\wedge p})_{\mfrakm})^{\wedge p} \to \dcomp{p}{(\dcomp{p}{M})_{\mfrakm}} \cong \dcomp{p}{M_{\mfrakm}}
    \end{equation*}
    is pure in \(\mcalD(S_{\mfrakm})\) for any graded maximal ideal \(\mfrakm\) of \(S\).

    (3) \(\Rightarrow\) (4): The condition (3) implies that \(\psi_{\mfrakm} \colon S_{\mfrakm} \to M_{\mfrakm}\) is pure in \(\mcalD(S_{\mfrakm})\), as in \cite{bhatt2024Perfectoid}*{Lemma 2.6}.

    (4) \(\Rightarrow\) (1): Writing \(M\) as a filtered colimit of perfect objects in \(\mcalD_{\graded{G}}(S)\), which are also perfect in \(\mcalD(S)\) by \cite{ishizuka2026Derived}*{Proposition 3.12}, we may assume that \(M\) is perfect in \(\mcalD_{\graded{G}}(S)\).
    Then \(M\) can be represented by a bounded complex \(P^{\bullet}\) of finitely generated graded projective \(S\)-modules.

    It suffices to show that \(\psi\) splits in \(\mcalD_{\graded{G}}(S)\). This is equivalent to saying that the evaluation morphism
    \begin{equation*}
        \ev \colon \underline{\Hom}_{\mcalD_{\graded{G}}(S)}(M, S) = H^0(R\underline{\Hom}_{\mcalD_{\graded{G}}(S)}(M, S)) \to S
    \end{equation*}
    of discrete \(G\)-graded \(S\)-modules is surjective, where \(R\underline{\Hom}_{\mcalD_{\graded{G}}(S)}(M, S)\) is the derived internal hom in \(\mcalD_{\graded{G}}(S) \simeq \mcalD(\Mod_{\graded{G}}(S))\).
    Since \(M \cong P^{\bullet}\) is a bounded complex of finitely generated graded projective \(S\)-modules, we have
    \begin{equation*}
        \underline{\Hom}_{\mcalD_{\graded{G}}(S)}(M, S) \simeq H^0(\underline{\Hom}_S(P^{\bullet}, S)) \xrightarrow{\simeq} H^0(\Hom_S(P^{\bullet}, S)) \simeq \Hom_{\mcalD(S)}(M, S),
    \end{equation*}
    where \(\underline{\Hom}_S(-, -)\) (resp., \(\Hom_S(-, -)\)) denotes the internal hom in \(\Mod_{\graded{G}}(S)\) (resp., in \(\Mod(S)\)), and the second isomorphism follows from \cite{ishizuka2026Derived}*{Proposition 2.9}.
    Therefore, the evaluation morphism \(\ev\) is identified with the evaluation morphism
    \begin{equation*}
        \ev' \colon \Hom_{\mcalD(S)}(M, S) \to S
    \end{equation*}
    in \(\mcalD(S)\). Since \(\psi_{\mfrakm}\) is pure in \(\mcalD(S_{\mfrakm})\) and \(M_{\mfrakm}\) is perfect in \(\mcalD(S_{\mfrakm})\) for any graded maximal ideal \(\mfrakm\) of \(S\), the morphism \(\psi_{\mfrakm}\) is in fact split in \(\mcalD(S_{\mfrakm})\). Thus, \(\ev'_{\mfrakm}\) is surjective for any graded maximal ideal \(\mfrakm\) of \(S\).
    This implies that \(\ev' \simeq \ev\) is surjective, and thus \(\psi\) is split in \(\mcalD_{\graded{G}}(S)\).
\end{proof}

Moreover, we developed the absolute perfectoidization of the structure sheaf of an adic formal scheme in \cite{ishizuka2026Algebraization}*{Section 4} and the following results are relevant to our main theorems.

\begin{theorem}[{\cite{ishizuka2026Algebraization}*{Definition 4.10, Corollary 4.11, and Theorem 4.28}}] \label{DerivedPushoutArcCohomology}
    Let \(\mscrX\) be an adic quasi-compact separated formal scheme such that \(p\) is topologically nilpotent in \(\mcalO_{\mscrX}\) and assume that \(\mscrX\) admits an ideal of definition \(\mcalI\) which is locally generated by a weakly proregular sequence.
    Then there exists a quasi-coherent \(\setE_{\infty}\)-\(\mcalO_{\mscrX}\)-algebra
    \begin{equation*}
        \mcalO_{\mscrX, \perfd} \in \CAlg(\mcalD_{\qcoh}(\mscrX))
    \end{equation*}
    called the \emph{absolute perfectoidization} of \(\mcalO_{\mscrX}\) such that, for any affine open subset \(\mscrU\) of \(\mscrX\), we have a functorial isomorphism
    \begin{equation*}
        R\Gamma(\mscrU, \mcalO_{\mscrX, \perfd}) \xrightarrow{\cong} \mcalO_{\mscrX}(\mscrU)_{\tperfd}
    \end{equation*}
    in \(\CAlg(\mcalD(\mcalO_{\mscrX}(\mscrU)))\).
\end{theorem}

\begin{proposition} \label{PropertiesGlobalPerfectoidization}
    Keep the setting of \Cref{DerivedPushoutArcCohomology}.
    Then \(\mcalO_{\mscrX, \perfd}\) satisfies the following properties.
    \begin{enumerate}
        \item \label{ComparisonOXPerfd} If \(\mscrX\) is bounded \(p\)-adic, then \(\mcalO_{\mscrX, \perfd}\) is isomorphic to the derived pushforward of the structure sheaf of the \emph{perfectization} \(\mscrX^{\pfd}\) of \(\mscrX\) in the sense of \cite{bhatt2025Aspects}.
        \item \label{GlobalPerfectoidizationSectionLimit} For any morphism \(f \colon \mscrY \to \mscrX\) of adic quasi-compact separated formal schemes such that \(\mscrY\) is perfectoid, the canonical morphism \(\mcalO_{\mscrX} \to Rf_*\mcalO_{\mscrY}\) uniquely factors through \(\mcalO_{\mscrX, \perfd}\) in \(\CAlg(\mcalD_{\qcoh}(\mscrX))\).
    \end{enumerate}
\end{proposition}


For a projective-type scheme \(X\), we can algebraize \(\mcalO_{\widehat{X}, \perfd}\) to an object of \(\mcalD_{\qcoh}(X)\), which is one of our main results in \cite{ishizuka2026Algebraization}. We recall it here in a special case for simplicity, which is sufficient for the applications in this paper.

\begin{theorem}[{\cite{ishizuka2026Algebraization}*{Theorem B and Theorem 5.10}}] \label{AlgebraizableOXperfd}
    Let \(X\) be a projective scheme over a Noetherian ring \(H^0(X, \mcalO_X)\) and let \(I\) be an ideal of \(H^0(X, \mcalO_X)\) containing \(p\).
    Then there exists an object \(\mcalO_{X,\perfd} \in \mcalD_{\qcoh}(X)\) together with an isomorphism
    \begin{equation*}
        \mcalO_{X, \perfd}^{\wedge} \xrightarrow{\cong} \mcalO_{\widehat{X}, \perfd}
    \end{equation*}
    in \(\mcalD(\widehat{X}_{\Zar}, \mcalO_{\widehat{X}})\), where \(\widehat{X}\) is the \(I\)-adic formal completion of \(X\) and \(\mcalO_{X, \perfd}^{\wedge}\) is the derived \(I\)-completion of \(\mcalO_{X, \perfd}\) (\Cref{SymMonDerivedCompletionFormalSchemes}).

    Indeed, this \(\mcalO_{X, \perfd}\) is the corresponding object \(\widetilde{R_{\grpfd}}\) of \(R_{\grpfd} \in \mcalD_{\graded{\setZ[1/p]}}(R)\) along the functor of associated sheaf \(\widetilde{(-)} \colon \mcalD_{\graded{\setZ[1/p]}}(R) \to \mcalD_{\qcoh}(X)\).
\end{theorem}

\section{Perfectoid splitting} \label{SectionPerfectoidSplitting}

We define global variants of lim-perfectoid purity and perfectoid purity introduced in \cite{bhatt2024Perfectoid}.


\begin{definition}[cf.~\cite{bhatt2025Aspects}] \label{DefDPS}
Let $X$ be a quasi-compact separated scheme and $\mscrX$ be the $p$-adic formal completion of $X$.
Assume that the closed subscheme defined by $p$ contains all closed points of $X$.
    \begin{enumerate}
    \item Following \cite{bhatt2024Perfectoid}*{\S 2.1}, a morphism \(f \colon \mcalF \to \mcalG\) in an \(\infty\)-category \(\mcalC\) admitting filtered colimits is \emph{ind-split} if there exists a filtered system \(\{f_\lambda \colon \mcalF \to \mcalG_{\lambda}\}_{\lambda \in \Lambda}\) in \(\mcalC\) whose filtered colimit is isomorphic to \(f\).
    \item We say that \(X\) is \emph{lim-perfectoid split} if the canonical morphism \(\mcalO_{\mscrX} \to \mcalO_{\mscrX, \perfd}\) in \(\mcalD_{\qcoh}(\mscrX)\) is ind-split.
    \item We say that \(X\) is \emph{perfectoid split} if there exists an affine morphism \(\pi \colon Y \to X\) such that the $p$-adic formal completion $\widehat{Y}$ is a perfectoid formal scheme and  \(\mcalO_{X} \to \pi_*\mcalO_{Y}\) in \(\QCoh(X)\) is ind-split.
    \end{enumerate}
\end{definition}


\begin{remark} \label{RemarkKodairaVanishing}
    Let \(X\) be a flat projective scheme over a \(p\)-torsion-free discrete valuation ring \(V\) whose residue characteristic is \(p\).
    Let \(L\) be an ample line bundle on \(X\).
    Let \(\mscrX\) be the \(p\)-adic formal completion of \(X\).
    In \cite{bhatt2025Aspects}*{Proposition 11.2.16}, Bhatt introduced the notion of \emph{absolute perfectoid purity} via the splitting property of \(\mcalO_{\mscrX} \to \mcalO_{\mscrX, \perfd}\) in \(\mcalD_{\qcoh}(\mscrX)\).
    Under this assumption, the author proved that the Kodaira vanishing \(H^{< \dim(X/V)}(X, L^{-1}) = 0\) holds when \(X\) is Cohen--Macaulay and \(X[1/p]\) is locally complete intersection.
    
    Lim-perfectoid splitting is a weaker condition, but the proof in \emph{loc.\ cit.} also works for an ind-split morphism \(\mcalO_{\mscrX} \to \mcalO_{\mscrX, \perfd}\). Therefore, the same Kodaira vanishing theorem holds for any lim-perfectoid split \(X\) such that \(X\) is Cohen--Macaulay and \(X[1/p]\) is locally complete intersection.
\end{remark}

We begin this subsection by establishing the fundamental properties of ind-split morphisms and  lim-perfectoid splitting.

\begin{proposition} \label{EquivAffineCase}
    Let \(R\) be a ring such that \(p\) is contained in the Jacobson radical of \(R\) and \(R\) has bounded \(p^\infty\)-torsion.
    \begin{enumerate}
        \item This ring \(R\) is lim-perfectoid pure if and only if \(\Spec(R)\) is lim-perfectoid split.
        \item This ring \(R\) is perfectoid pure if and only if \(\Spec(R)\) is perfectoid split.
    \end{enumerate}
\end{proposition}

\begin{proof}
    We equip \(R\) with the \(p\)-adic topology and consider the \(p\)-adic affine formal spectrum \(\Spf(R)\).
    Under the categorical equivalence \(\mcalD_{\qcoh}(\Spf(R)) \xrightarrow{\simeq} \mcalD^{\comp{p}}(R)\) given by \(R\Gamma(\Spf(R), -)\) (\cite{ishizuka2026Algebraization}*{Proposition C.10}), the absolute perfectoidization \(\mcalO_{\Spf(R), \perfd}\) corresponds to \(R_{\perfd}\).
    Therefore, \(R \to R_{\perfd}\) is ind-split in \(\mcalD^{\comp{p}}(R)\) if and only if \(\mcalO_{\Spf(R)} \to \mcalO_{\Spf(R), \perfd}\) is ind-split in \(\mcalD_{\qcoh}(\Spf(R))\). This proves (1).

    If \(R\) is perfectoid pure, then \(\Spec(R)\) is perfectoid split. Conversely, assume there exists an affine morphism \(\pi \colon Y \to \Spec(R)\) from a scheme \(\mscrY\) such that \(\widehat{Y}\) is a perfectoid formal scheme and the induced morphism \(R \to \pi_*\mcalO_{Y}\) is ind-split.
    Since \(\pi\) is affine, \(\widehat{Y}\) is an affine formal scheme; thus, taking the global sections shows that \(R\) is perfectoid pure. This proves (2).
\end{proof}

\begin{proposition}[cf.~\Cref{CICase}] \label{GPPtoGLPP}
    Let \(X\) be a quasi-compact separated scheme such that the closed subscheme defined by \(p\) contains all closed points of \(X\).
    If \(X\) is perfectoid split, then it it lim-perfectoid split.
\end{proposition}

\begin{proof}
    Take an affine morphism \(\pi \colon Y \to X\) from a scheme \(Y\) such that the \(p\)-adic formal completion \(\widehat{Y}\) is a perfectoid formal scheme and \(\mcalO_{X} \to \pi_*\mcalO_{Y}\) is ind-split.
    By \Cref{PropertiesGlobalPerfectoidization}\Cref{GlobalPerfectoidizationSectionLimit}, the \(p\)-adic completion of this morphism factors through the canonical morphism \(\mcalO_{\widehat{X}} \to \mcalO_{\widehat{X}, \perfd}\).
    Since \(\mcalD(\widehat{X}_{\Zar}, \mcalO_{\widehat{X}})\) is a stable \(\infty\)-category, if a composite morphism is ind-split, then the first morphism is also ind-split, as shown in \cite{bhatt2024Perfectoid}*{Lemma 2.6}.
    This implies that \(X\) is lim-perfectoid split.
\end{proof}

\begin{proposition}{(cf.~\cites{bhatt2023Globally, takamatsu2023Minimal})} \label{PlusRegular}
    Let \(X\) be a quasi-compact separated excellent integral scheme such that the closed subscheme defined by \(p\) contains all closed points of \(X\).
    If \(X\) is globally \pmb{+}-regular, then \(X\) is perfectoid split.
\end{proposition}

\begin{proof}
    Choose an absolute integral closure \(\pi \colon X^+ \to X\) of \(X\), which is an affine morphism.
    By \cite{bhatt2023Globally}*{Definition 6.1 and Convention 4.1}, if \(X\) is globally \pmb{+}-regular, then the canonical morphism
    \begin{equation*}
        \mcalO_X \to \pi_*\mcalO_{X^+}
    \end{equation*}
    is ind-split in \(\QCoh(X)\). Taking the \(p\)-adic formal completion, \(\widehat{X^+}\) becomes a perfectoid formal scheme since any affine subscheme of \(X^+\) is absolute integrally closed. So \(X\) is perfectoid split.
\end{proof}

The following result was obtained through discussions with Joe Waldron.

\begin{proposition}\label{ind-split-dual}
Let $(R,\m)$ be a Noetherian local ring admitting a dualizing complex. 
Let $I \subseteq \m$ be an ideal of $R$.
Let $X$ be a proper scheme over $\Spec R$ and let \(\mscrX\) be the \(I\)-adic formal completion of \(X\). 
Let $f \colon \cO_X \to \mathcal{F}$ be a morphism in $\mcalD_{\qcoh}(X)$.
Then the following are equivalent:
\begin{enumerate}
    \item $f^{\wedge}$ is ind-split in \(\mcalD_{\qcoh}(\mscrX)\), where \((-)^{\wedge}\) is the derived \(I\)-completion defined in \Cref{SymMonDerivedCompletionFormalSchemes}.
    \item $f$ is ind-split in \(\mcalD_{\qcoh}(X)\).
    \item The canonical morphism
    \begin{equation*}
            H^d_{\m}(\mscrX, \omega_{\mscrX}) \to H^d_{\m}(\mscrX, \widehat{\mcalF} \otimes^L_{\mcalO_{\mscrX}} \omega_{\mscrX}^{\bullet}[-d])
    \end{equation*}
        is injective.
    \item The canonical morphism
        \begin{equation*}
            H^d_{\m}(X, \omega_X) \to H^d_{\m}(X, \mcalF \otimes^L_{\mcalO_X} \omega_X^{\bullet}[-d])
        \end{equation*}
        is injective.
\end{enumerate}
\end{proposition}

\begin{proof}
    (1) \(\Rightarrow\) (3): If $f^{\wedge}$ is ind-split, then so is
    \[
    \omega_{\mscrX}^{\bullet}[-d] \to \widehat{\mcalF} \otimes^L \omega_{\mscrX}^{\bullet}[-d].
    \]
    Taking $H^0_{\m}(-)$, we obtain the injection
    \begin{equation*}
        H^d_{\m}(\mscrX, \omega^{\bullet}_{\mscrX}[-d]) \to H^d_{\m}(\mscrX, \widehat{\mcalF} \otimes^L_{\mcalO_{\mscrX}} \omega^{\bullet}_{\mscrX}[-d]).
    \end{equation*}
    As in the proof of \Cref{matlis-dual}, the target of this morphism is \(H^d_{\m}(\mscrX, \omega_{\mscrX})\).

    (2) \(\Rightarrow\) (4): If $f$ is ind-split, the same argument as above gives the desired injection.

    (3) \(\Leftrightarrow\) (4): It suffices to show that the morphism
    \begin{equation*}
        R\Gamma(\mscrX, \omega_{\mscrX}^{\bullet}) \to R\Gamma(\mscrX, \widehat{\mcalF} \otimes^L_{\mcalO_{\mscrX} \omega_{\mscrX}^{\bullet}})
    \end{equation*}
    in \(\mcalD(R)\) is the derived \(I\)-completion of the morphism
    \begin{equation*}
        R\Gamma(X, \omega_X^{\bullet}) \to R\Gamma(X, \mcalF \otimes^L_{\mcalO_X} \omega_X^{\bullet})
    \end{equation*}
    in \(\mcalD(R)\).
    Since \(I\cO_X\) is globally generated by \(I \subseteq R\), we can apply \Cref{DerivedCompletionFormalSchemes} and \Cref{SymMonDerivedCompletionFormalSchemes} to this setting, which concludes the proof.

    (2) \(\Rightarrow\) (1): It follows from the fact that the functor \((-)^{\wedge} \colon \mcalD_{\qcoh}(X) \to \mcalD_{\qcoh}(\mscrX)\) preserves filtered colimits (\Cref{SymMonDerivedCompletionFormalSchemes}). 

    (4) \(\Rightarrow\) (2): We can take a filtered system $\{\cO_{X} \to \mcalF_{\lambda}\}_{\lambda \in \Lambda}$ of perfect complexes \(\mcalF_{\lambda} \in \mcalD_{\coh}^b(X)\) such that $\colim_{\lambda \in \Lambda} \mcalF_{\lambda} \simeq \mcalF$ in $\mcalD_{\qcoh}(X)$.
    Thus, the isomorphism
    \begin{equation*}
        R\Gamma(X, \mcalF \otimes^L_{\mcalO_X} \omega_X^{\bullet}) \cong \colim_{\lambda \in \Lambda} R\Gamma(X, \mcalF_{\lambda} \otimes^L_{\mcalO_X} \omega_X^{\bullet})
    \end{equation*}
    holds in \(\mcalD(R)\).
    Shifting by \([-d]\) and applying \(H^d_{\m}(-)\), our assumption yields the injection
    \begin{equation*}
        H^0_{\m}(X, \omega_X^{\bullet}) \cong H^d_{\m}(X,\omega_{X}) \to H^d_{\m}(X, \mcalF_{\lambda} \otimes^L_{\mcalO_{X}} \omega_{X}^{\bullet}[-d]) \cong H^0_{\m}(X, \mcalF_{\lambda} \otimes^L_{\mcalO_X} \omega_X^{\bullet})
    \end{equation*}
    for any \(\lambda \in \Lambda\).
    Taking the Matlis dual over $R$ and invoking \cref{matlis-dual}, we obtain the surjection
    \begin{equation*}
        \comp{\m}{\Hom_{\mcalD_{\qcoh}(X)}(\mcalF_{\lambda},\cO_{X})} \to \comp{\m}{H^0(X,\cO_{X})}
    \end{equation*}
    of \(R\)-modules for any \(\lambda \in \Lambda\).
    Since these Hom-modules are finitely generated \(R\)-modules, the map is surjective even before completion.
    Therefore, we obtain that $\cO_{X} \to \mcalF_{\lambda}$ splits in \(\mcalD_{\qcoh}(X)\), as desired.
\end{proof}

\begin{corollary} \label{CompareGLPPGLPI}
    Let $(R,\m)$ be an Noetherian local ring of residue characteristic $p>0$ admitting a dualizing complex.
    Let $I \subseteq \m$ be an ideal of $R$ containing $p$.
    Let $X$ be a projective scheme over $\Spec R$. 
    Then the following are equivalent:
    \begin{enumerate}
        \item \(X\) is lim-perfectoid split, i.e., the canonical morphism \(\mcalO_{\mscrX} \to \mcalO_{\mscrX, \perfd}\) is ind-split in \(\mcalD_{\qcoh}(\mscrX)\).
        \item The morphism \(\mcalO_X \to \mcalO_{X, \perfd}\) defined in \Cref{AlgebraizableOXperfd} is ind-split in \(\mcalD_{\qcoh}(X)\).
        \item The canonical morphism
        \begin{equation*}
            H^d_{\m}(\mscrX, \omega_{\mscrX}) \to H^d_{\m}(\mscrX, \mcalO_{\mscrX, \perfd} \otimes^L_{\mcalO_{\mscrX}} \omega_{\mscrX}^{\bullet}[-d])
        \end{equation*}
        is injective.
        \item The canonical morphism
        \begin{equation*}
            H^d_{\m}(X, \omega_X) \to H^d_{\m}(X, \mcalO_{X, \perfd} \otimes^L_{\mcalO_X} \omega_X^{\bullet}[-d])
        \end{equation*}
        is injective.
    \end{enumerate}
\end{corollary}

\begin{proof}
It follows from \Cref{ind-split-dual} and \Cref{AlgebraizableOXperfd}.
\end{proof}


\begin{proposition}\label{descent-derived-split}
Let $f \colon Y \to X$ be a morphism of quasi-compact separated schemes whose closed subschemes defined by \(p\) contains all closed points of them.
Assume that the canonical morphism $\cO_{X} \to Rf_*\cO_{Y}$ is ind-split.
If $Y$ is lim-perfectoid split, then so is $X$.
\end{proposition}

\begin{proof}
Set the \(p\)-adic formal completions \(\mscrX \defeq \widehat{X}\) and \(\mscrY \defeq \widehat{Y}\).
Consider the following commutative diagram:
\[
\begin{tikzcd}
    \cO_{\mscrX} \arrow[r] \arrow[d] & \cO_{\mscrX,\perfd} \arrow[d] \\
    Rf_*\cO_{\mscrY} \arrow[r] & Rf_*\cO_{\mscrY,\perfd}.
\end{tikzcd}
\]
Since $Y$ is lim-perfectoid split, the bottom horizontal morphism is ind-split.
By assumption, the composition
\[
\cO_{\mscrX} \to \cO_{\mscrX,\perfd} \to Rf_*\cO_{\mscrY,\perfd}
\]
is ind-split.
Therefore, $X$ is lim-perfectoid split.
\end{proof}

\section{Gradedwise completed local cohomology} \label{SectionLocalCoho}

\subsection{Gradedwise completed local cohomology}

\begin{definition} \label{GradedwiseCompletedLocCoho}
    Let \(R\) be a \(G\)-graded adic ring with a homogeneous ideal of definition \(I\), and let \(\mfrakm\) be a homogeneous ideal of \(R\).
    Let \(M\) be a \(G\)-graded \(R\)-module.
    We define the \emph{gradedwise completed local cohomology} of \(M\) with respect to \(\mfrakm\), denoted by \(R\Gamma_{\mfrakm}^{\wedge}(M)\), as
    \begin{equation*}
        R\Gamma_{\mfrakm}^{\wedge}(M) \defeq \grlim_{n \geq 1} R\Gamma_{\mfrakm}(M/I^nM) = \bigoplus_{g \in G} R\lim_{n \geq 1} R\Gamma_{\mfrakm}(M/I^nM)_g
    \end{equation*}
    in \(\mcalD_{\graded{G}}(R)\).
    By definition, this construction is functorial in \(M\), and the canonical morphism \(R\Gamma_{\mfrakm}^{\wedge}(M) \to R\Gamma_{\mfrakm}^{\wedge}(\grcomp{I}{M})\) is an isomorphism.
\end{definition}

\begin{lemma} \label{GradedwiseCompletedCech}
    Let \(R\) be a \(G\)-graded adic ring with a homogeneous ideal of definition \(I\), and let \(\mfrakm\) be a homogeneous ideal of \(R\).
    Let \(M\) be a \(G\)-graded \(R\)-module.
    Assume that \(\mfrakm\) is finitely generated by homogeneous elements \(x_1, \dots, x_r \in R\).
    Then the gradedwise completed local cohomology \(R\Gamma_{\mfrakm}^{\wedge}(M)\) can be represented by the complex
    \begin{equation*}
        0 \to \grcomp{I}{M} \to \bigoplus_{i = 1}^r \grcomp{I}{M_{x_i}} \to \bigoplus_{1 \leq i < j \leq r} \grcomp{I}{M_{x_ix_j}} \to \cdots \to \grcomp{I}{M_{x_1 \cdots x_r}} \to 0
    \end{equation*}
    where the term \(\grcomp{I}{M}\) sits in cohomological degree \(0\). In other words, this complex is the gradedwise completion of the usual augmented \v{C}ech complex computing \(R\Gamma_{\mfrakm}(M)\).
\end{lemma}

\begin{proof}
    Using the generators \(x_1, \dots, x_r\) of \(\mfrakm\), we know that \(R\Gamma_{\mfrakm}(M/I^nM)\) is computed by the augmented \v{C}ech complex \(\cech_{aug}(\underline{x}; M/I^nM) = \cech_{aug}(x_1, \dots, x_r; M/I^nM) \defeq (0 \to M/I^nM \to \cech(\underline{x}; M/I^nM))\), and its degree \(g\)-part is \(R\Gamma_{\mfrakm}(M/I^nM)_g = \cech_{aug}(\underline{x}; M/I^nM)_g\).
    We will compute \(R\lim_{n \geq 1} R\Gamma_{\mfrakm}(M/I^nM)_g\) for each \(g \in G\).
    The transition morphisms \(\cech_{aug}(\underline{x}; M/I^{n+1}M)_g \to \cech_{aug}(\underline{x}; M/I^nM)_g\) give rise to inverse systems in each cohomological degree, for example,
    \begin{equation*}
        \cdots \to \bigoplus_{i=1}^r (M/I^{n+1}M)_{x_i, g} \to \bigoplus_{i=1}^r (M/I^nM)_{x_i, g} \to \cdots
    \end{equation*}
    in cohomological degree \(1\), where \((M/I^nM)_{x_i, g}\) is the degree \(g\)-part of the localization \((M/I^nM)_{x_i}\) of the \(G\)-graded \(R\)-module \(M/I^nM\).
    Here, the \(R^1\lim_{n \geq 1}\) of these inverse systems vanishes, and the limit is isomorphic to the direct sum of the \(I\)-adic completions, e.g., \(\lim_{n \geq 1}\bigoplus_{i=1}^r (M/I^nM)_{x_i, g} \cong \bigoplus_{i=1}^r (\grcomp{I}{M_{x_i}})_g\).
    Using \citeStaN{07KW}{5} and \citeSta{07KX}, we have
    \begin{align*}
        & R\lim_{n \geq 1}R\Gamma_{\mfrakm}(M/I^nM)_g = R\lim_{n \geq 1}\cech_{aug}(\underline{x}; M/I^nM)_g \\
        & \cong (0 \to \grcomp{I}{M}_g \to \bigoplus_{i=1}^r (\grcomp{I}{M_{x_i}})_g \to \bigoplus_{1 \leq i < j \leq r} (\grcomp{I}{M_{x_ix_j}})_g \to \cdots \to (\grcomp{I}{M_{x_1 \cdots x_r}})_g \to 0)
    \end{align*}
    and this gives the desired representation since the direct sum \(\bigoplus_{g \in G} (\grcomp{I}{M_{x_i}})_g\) is the \(I\)-adic gradedwise completion \(\grcomp{I}{M_{x_i}}\) of the \(G\)-graded \(R\)-module \(M_{x_i}\).
\end{proof}

\subsection{Local cohomology of derived graded modules}

\begin{construction} \label{DefLocalCohoGradedModules}
    Let \(R\) be a \(G\)-graded ring and let \(M\) be an object of \(\mcalD_{\graded{G}}(R)\).
    Let \(f \in R\) be a homogeneous element.
    Then \emph{the local cohomology \(R\Gamma_{(f)}(M)\) of \(M\) with respect to \(f\)} is defined via the fiber sequence
    \begin{equation*}
        R\Gamma_{(f)}(M) \to M \to M[1/f]
    \end{equation*}
    in \(\mcalD_{\graded{G}}(R)\), where \(M[1/f] = M \otimes^L_{R} R[1/f]\) is the localization of \(M\) in \(\mcalD_{\graded{G}}(R)\).
    For a finitely generated homogeneous ideal \(\mfrakm\) of \(R\) with homogeneous generators \(f_1, \dots, f_r\), \emph{the local cohomology \(R\Gamma_{(f_1, \dots, f_r)}(M)\) of \(M\) with respect to \((f_1, \dots, f_r)\)} is defined as the composition
    \begin{equation*}
        R\Gamma_{(f_1, \dots, f_r)}(M) \defeq (R\Gamma_{(f_r)} \circ \cdots \circ R\Gamma_{(f_1)})(M) \in \mcalD_{\graded{G}}(R).
    \end{equation*}
    As we will show in \Cref{LocCohoWellDef} below, this definition is independent of the choice of homogeneous generators. Thus, we can unambiguously define \emph{the local cohomology \(R\Gamma_{\mfrakm}(M)\) of \(M\) with respect to \(\mfrakm\)}.
\end{construction}

\begin{lemma} \label{LocCohoWellDef}
    In the above setting, the following hold:
    \begin{enumerate}
        \item \label{LocCohoWellDefUnderlying} The underlying \(R\)-module of \(R\Gamma_{(f_1, \dots, f_r)}(M)\) coincides with the classical local cohomology of the underlying \(R\)-module \(M\) with respect to \(\mfrakm = (f_1, \dots, f_r)\).
        \item \label{LocCohoWellDefWellDefined} The local cohomology \(R\Gamma_{\mfrakm}(M) \in \mcalD_{\graded{G}}(R)\) is well-defined; i.e., it is independent of the choice of homogeneous generators of \(\mfrakm\).
        \item \label{LocCohoWellDefLocalization} For any homogeneous element \(f \in R\), the canonical morphism \(R\Gamma_{\mfrakm}(M)[1/f] \to R\Gamma_{\mfrakm}(M[1/f])\) is an isomorphism in \(\mcalD_{\graded{G}}(R)\).
        \item \label{LocCohoWellDefIsomorphism} There exists a canonical isomorphism \(R\Gamma_{\mfrakm}(M) \cong R\Gamma_{\mfrakm}(R) \Lgrotimes_R M\) in \(\mcalD_{\graded{G}}(R)\).
        \item \label{LocCohoWellDefColimits} The functor \(R\Gamma_{\mfrakm}(-) \colon \mcalD_{\graded{G}}(R) \to \mcalD_{\graded{G}}(R)\) preserves all colimits.
    \end{enumerate}
\end{lemma}

\begin{proof}
    \Cref{LocCohoWellDefUnderlying}: The fiber sequence \(R\Gamma_{(f)}(M) \to M \to M[1/f]\) in \(\mcalD_{\graded{G}}(R)\) maps to the corresponding fiber sequence in \(\mcalD(R)\) via the forgetful functor, by \cite{ishizuka2026Derived}*{Proposition 3.6(7)}.
    The fiber of \(M \to M[1/f]\) in \(\mcalD(R)\) is naturally identified with the classical local cohomology of \(M\) with respect to \((f)\), which is represented by the \v{C}ech complex \(\cech(f; M) \simeq \cech(f; R) \otimes^L_R M\).
    Iterating this process, we see that the underlying \(R\)-module of \(R\Gamma_{(f_1, \dots, f_r)}(M)\) coincides with the classical local cohomology of \(M\) with respect to \(\mfrakm\).

    \Cref{LocCohoWellDefWellDefined}: We next prove the well-definedness of \(R\Gamma_{\mfrakm}(-)\).
    Let \(f_1, \dots, f_r\) and \(g_1, \dots, g_s\) be two sets of homogeneous generators of \(\mfrakm\).
    Applying the functor \(R\Gamma_{(f_1, \dots, f_r)}\) to the canonical morphism \(R\Gamma_{(g_1, \dots, g_s)}(M) \to M\), we obtain a morphism
    \begin{equation*}
        R\Gamma_{(f_1, \dots, f_r)}(R\Gamma_{(g_1, \dots, g_s)}(M)) \to R\Gamma_{(f_1, \dots, f_r)}(M)
    \end{equation*}
    in \(\mcalD_{\graded{G}}(R)\). By \Cref{LocCohoWellDefUnderlying}, its image in \(\mcalD(R)\) is an isomorphism, because it corresponds to the canonical equivalence of \v{C}ech complexes \(\cech(\underline{f}; \cech(\underline{g}; M)) \xrightarrow{\sim} \cech(\underline{f}; M)\) associated with the equality of ideals \((f_1, \dots, f_r) = (g_1, \dots, g_s)\). By the conservativity of the forgetful functor (\cite{ishizuka2026Derived}*{Proposition 3.6(7)}), the morphism is an isomorphism in \(\mcalD_{\graded{G}}(R)\). To complete the proof of well-definedness, it suffices to show that the functors \(R\Gamma_{(f_1, \dots, f_r)}\) and \(R\Gamma_{(g_1, \dots, g_s)}\) commute up to natural isomorphism. By definition, we can reduce this to showing that \(R\Gamma_{(f)} \circ R\Gamma_{(g)} \simeq R\Gamma_{(g)} \circ R\Gamma_{(f)}\) for any homogeneous elements \(f\) and \(g\) of \(R\).

    We have fiber sequences
    \begin{align*}
        R\Gamma_{(f)}(R\Gamma_{(g)}(M)) & \to R\Gamma_{(g)}(M) \to R\Gamma_{(g)}(M)[1/f] \\
        R\Gamma_{(g)}(R\Gamma_{(f)}(M)) & \to R\Gamma_{(g)}(M) \to R\Gamma_{(g)}(M[1/f])
    \end{align*}
    in \(\mcalD_{\graded{G}}(R)\), where the former comes from the definition of \(R\Gamma_{(f)}(R\Gamma_{(g)}(M))\), and the latter is obtained by applying the functor \(R\Gamma_{(g)}\) to the fiber sequence defining \(R\Gamma_{(f)}(M)\) (this is valid since \(R\Gamma_{(g)}\), being defined as a fiber, preserves fiber sequences).
    It therefore suffices to establish a canonical isomorphism \(R\Gamma_{(g)}(M)[1/f] \xrightarrow{\cong} R\Gamma_{(g)}(M[1/f])\). This follows from the commutativity of the following diagram
    \begin{center}
        \begin{tikzcd}
            {R\Gamma_{(g)}(M)[1/f]} \arrow[d] \arrow[r] & {M[1/f]} \arrow[d, Rightarrow, no head] \arrow[r] & {(M[1/g])[1/f]} \arrow[d, "\cong"] \\
            {R\Gamma_{(g)}(M[1/f])} \arrow[r]           & {M[1/f]} \arrow[r]                                & {(M[1/f])[1/g]}                   
        \end{tikzcd}
    \end{center}
    whose horizontal sequences are fiber sequences in \(\mcalD_{\graded{G}}(R)\).

    \Cref{LocCohoWellDefLocalization}: This reduces to the case \(\mfrakm = (g)\), which was established above.

    \Cref{LocCohoWellDefIsomorphism}: Applying \(- \Lgrotimes_R M\) to the fiber sequence \(R\Gamma_{(f)}(R) \to R \to R[1/f]\) in \(\mcalD_{\graded{G}}(R)\) yields a fiber sequence
    \begin{equation*}
        R\Gamma_{(f)}(R) \Lgrotimes_R M \to M \to R[1/f] \Lgrotimes_R M \cong M[1/f]
    \end{equation*}
    in \(\mcalD_{\graded{G}}(R)\), which naturally identifies \(R\Gamma_{(f)}(R) \Lgrotimes_R M\) with \(R\Gamma_{(f)}(M)\).
    Using this conclusion inductively, we obtain the following isomorphism in \(\mcalD_{\graded{G}}(R)\):
    \begin{align*}
        R\Gamma_{(f_1, \dots, f_{r-1})}(R\Gamma_{(f_r)}(M)) & \cong R\Gamma_{(f_1, \dots, f_{r-1})}(R\Gamma_{(f_r)}(R) \Lgrotimes_R M) \\
        & \cong R\Gamma_{(f_1, \dots, f_{r-1})}(R) \Lgrotimes_R R\Gamma_{(f_r)}(R) \Lgrotimes_R M \cong R\Gamma_{(f_1, \dots, f_r)}(R) \Lgrotimes_R M
    \end{align*}
    as desired.

    \Cref{LocCohoWellDefColimits}: This reduces to the case \(\mfrakm = (f)\), which is immediate because \(R\Gamma_{(f)}(M)\) is defined as the fiber of \(M \to M[1/f]\), and both the identity functor and localization \((-)[1/f]\) preserve all colimits.
\end{proof}

We will use the following lemma to compare the gradedwise completed local cohomology with the limit of the local cohomology of the graded quotients by powers of \(f\).

\begin{lemma} \label{BoundedProIsom}
    Let \(R\) be a ring and let \(M\) be an \(R\)-module.
    Let \(f \in R\), and assume that the \(f^{\infty}\)-torsion of \(M\) is bounded.
    Then the morphism of pro-systems
    \begin{equation*}
        \{M/^Lf^n\}_{n \geq 1} \to \{M/f^nM\}_{n \geq 1}
    \end{equation*}
    induced by the canonical morphisms \(M/^Lf^n \to M/f^nM\) is a pro-isomorphism.
\end{lemma}

\begin{proof}
    See the proof of \cite{bhatt2018Integral}*{Lemma 4.7(2)}.
\end{proof}

\begin{proposition} \label{CompatibleDerivedLocalCoho}
    Let \(R\) be an \(I\)-adic \(G\)-graded ring.
    Let \(M\) be a \(G\)-graded \(R\)-module.
    Let \(f \in R\) be a homogeneous element such that the \(f^\infty\)-torsion of \(M\) is bounded.
    Let \(\mfrakm\) be a homogeneous ideal of \(R\).
    Then the canonical morphism
    \begin{equation*}
        \grlim_{n \geq 1}R\Gamma_{\mfrakm}(M/^L f^n) = \bigoplus_{g \in G} R\lim_{n \geq 1} (R\Gamma_{\mfrakm}(M/^L f^n)_g) \xrightarrow{\cong} \bigoplus_{g \in G} R\lim_{n \geq 1} (R\Gamma_{\mfrakm}(M/f^nM)_g) = R\Gamma_{\mfrakm}^{\wedge}(M)
    \end{equation*}
    is an isomorphism in \(\mcalD_{\graded{G}}(R)\), where the right-hand side is the gradedwise completed local cohomology defined in \Cref{GradedwiseCompletedLocCoho}.
\end{proposition}

\begin{proof}
    By \Cref{BoundedProIsom}, the assumption that \(M\) has bounded \(f^\infty\)-torsion implies that, for any \(n \geq 1\), there exists \(m > n\) such that the left vertical morphism in the following commutative diagram of fiber sequences is zero in \(\mcalD_{\graded{G}}(R)\):
    \begin{center}
        \begin{tikzcd}
            K_m \arrow[r] \arrow[d] & M/^L f^m \arrow[r] \arrow[d] & M/f^mM \arrow[d, two heads] \\
            K_n \arrow[r]           & M/^L f^n \arrow[r]           & M/f^nM                     
        \end{tikzcd}
    \end{center}
    Applying \(R\Gamma_{\mfrakm}(-)\), we conclude that the pro-systems \(\{R\Gamma_{\mfrakm}(M/^L f^n)\}_{n \geq 1}\) and \(\{R\Gamma_{\mfrakm}(M/f^nM)\}_{n \geq 1}\) are pro-isomorphic.
    Taking the limits in \(\mcalD_{\graded{G}}(R)\), we obtain the desired isomorphism.
\end{proof}

\begin{lemma} \label{PropertiesCompLocCoho}
    Let $R$ be an \(I\)-adic graded ring and let \(M\) be an object of \(\mcalD_{\graded{G}}(R)\).
    Let \(f_1, \ldots, f_r\) be a fixed set of homogeneous generators of \(I\), and let \(\mfrakm\) and \(\mfrakn\) be homogeneous ideals of \(R\).
    \begin{enumerate}
        \item \label{PropertiesCompLocCohoRep} There exists a canonical isomorphism
        \begin{equation*}
            \grlim_{n \geq 1} R\Gamma_{\mfrakm}(M/^L(f_1^n, \dots, f_r^n)) \cong \dgrcomp{I}{R\Gamma_{\mfrakm}(M)}
        \end{equation*}
        in \(\mcalD_{\graded{G}}(R)\).
        \item \label{PropertiesCompLocCohoBounded} In the case of \(I = (f)\), if \(M\) is concentrated in degree \(0\) and has bounded \(f^\infty\)-torsion, then
        \[
        \dgrcomp{f}{R\Gamma_{\mfrakm}(M)}=R\Gamma^{\wedge}_{\mfrakm}(M),
        \]
        where $R\Gamma^{\wedge}_{\mfrakm}(M)$ is defined in \Cref{GradedwiseCompletedLocCoho} and it has a \v{C}ech representation in \Cref{GradedwiseCompletedCech}.
        \item \label{PropertiesCompLocCohoFiberSeq} There exists a fiber sequence
        \[
        \dgrcomp{I}{R\Gamma_{(f)}(M)} \to \dgrcomp{I}{M} \to \dgrcomp{I}{M[1/f]}
        \]
        in \(\mcalD_{\graded{G}}(R)\) for any homogeneous element \(f \in R\).
        \item \label{PropertiesCompLocCohoLocalization} For any homogeneous element \(f \in R\), the canonical morphism 
        \[
        \dgrcomp{I}{\dgrcomp{I}{R\Gamma_{\mfrakm}(M)}[1/f]} \to \dgrcomp{I}{R\Gamma_{\mfrakm}(M[1/f])}
        \]
        is an isomorphism in \(\mcalD_{\graded{G}}(R)\).
        \item \label{PropertiesCompLocCohoComposition} We have a canonical isomorphism \(\dgrcomp{I}{R\Gamma_{\mfrakm}(\dgrcomp{I}{R\Gamma_{\mfrakn}(M)})} \cong \dgrcomp{I}{R\Gamma_{\mfrakn + \mfrakm}(M)}\) in \(\mcalD_{\graded{G}}(R)\).
        \item \label{PropertiesCompLocCohoStable} The canonical morphism \(M \to \dgrcomp{I}{M}\) induces an isomorphism \(R\Gamma_{I}(M) \xrightarrow{\cong} R\Gamma_{I}(\dgrcomp{I}{M})\) in \(\mcalD_{\graded{G}}(R)\).
        \item \label{PropertiesCompLocCohoCompLoc} The canonical morphism \(R\Gamma_{I}(M) \to M\) induces an isomorphism \(\dgrcomp{I}{R\Gamma_{I}(M)} \xrightarrow{\cong} \dgrcomp{I}{M}\) in \(\mcalD_{\graded{G}}(R)\).
    \end{enumerate}
\end{lemma}

\begin{proof}
    \Cref{PropertiesCompLocCohoRep}: By \Cref{LocCohoWellDef}\Cref{LocCohoWellDefIsomorphism} and \cite{ishizuka2026Derived}*{Lemma 4.8(5)}, we have isomorphisms
    \begin{align*}
        \grlim_{n \geq 1} R\Gamma_{\mfrakm}(M/^L(f_1^n, \dots, f_r^n)) & \cong \grlim_{n \geq 1} (R\Gamma_{\mfrakm}(R) \Lgrotimes_{R} (M/^L(f_1^n, \dots, f_r^n))) \\
        & \cong \grlim_{n \geq 1} (R\Gamma_{\mfrakm}(R) \Lgrotimes_R (M \Lgrotimes_R R/^L(f_1^n, \dots, f_r^n))) \\
        & \cong \grlim_{n \geq 1} (R\Gamma_{\mfrakm}(M)/^L(f_1^n, \dots, f_r^n)) \cong \dgrcomp{I}{R\Gamma_{\mfrakm}(M)}
    \end{align*}
    in \(\mcalD_{\graded{G}}(R)\), where the second isomorphism follows from \cite{ishizuka2026Derived}*{(4.1)}.

    \Cref{PropertiesCompLocCohoBounded}: This follows from \Cref{PropertiesCompLocCohoRep} and \Cref{CompatibleDerivedLocalCoho}.

    \Cref{PropertiesCompLocCohoFiberSeq}: This follows from the definition of \(R\Gamma_{(f)}\).

    \Cref{PropertiesCompLocCohoLocalization}:
    First, we have a natural isomorphism
    \begin{equation*}
        \dgrcomp{I}{R\Gamma_\mfrakm(M)[1/f]} \xrightarrow{\sim} \dgrcomp{I}{\dgrcomp{I}{R\Gamma_\mfrakm(M)}[1/f]}.
    \end{equation*}
    Therefore, the assertion follows from \Cref{LocCohoWellDef}\Cref{LocCohoWellDefLocalization}.

    \Cref{PropertiesCompLocCohoComposition}: Using \Cref{LocCohoWellDef}\Cref{LocCohoWellDefIsomorphism} and \cite{ishizuka2026Derived}*{Lemma 4.8(4)}, we have the following isomorphisms
    \begin{align*}
        \dgrcomp{I}{R\Gamma_{\mfrakm}(\dgrcomp{I}{R\Gamma_{\mfrakn}(M)})} & \cong \dgrcomp{I}{R\Gamma_{\mfrakm}(R) \Lgrotimes_{R} \dgrcomp{I}{R\Gamma_{\mfrakn}(M)}} \\
        & \cong \dgrcomp{I}{R\Gamma_{\mfrakm}(R) \Lgrotimes_{R} \dgrcomp{I}{R\Gamma_{\mfrakn}(R) \Lgrotimes_{R} M}} \\
        & \cong \dgrcomp{I}{R\Gamma_{\mfrakm}(R) \Lgrotimes_{R} R\Gamma_{\mfrakn}(R) \Lgrotimes_{R} M} \cong \dgrcomp{I}{R\Gamma_{\mfrakm + \mfrakn}(M)}
    \end{align*}
    in \(\mcalD_{\graded{G}}(R)\).

    \Cref{PropertiesCompLocCohoStable}: If we can establish the isomorphism \(R\Gamma_I(M) \xrightarrow{\cong} R\Gamma_I(\dgrcomp{f}{M})\) for any homogeneous element \(f \in I\), the general case follows by iterating this procedure.
    Moreover, since \(R\Gamma_I(M) \cong R\Gamma_{I}(R\Gamma_{(f)}(M))\) for \(f \in I\), it suffices to consider the case \(I = (f)\) and prove that \(R\Gamma_{(f)}(M) \xrightarrow{\cong} R\Gamma_{(f)}(\dgrcomp{f}{M})\).
    Consider the following commutative diagram, whose rows are fiber sequences:
    \begin{equation}\label{local-comp}
        \begin{tikzcd}
            R\Gamma_{(f)}(M) \arrow[d] \arrow[r]    & M \arrow[d] \arrow[r]    & {M[1/f]} \arrow[d]    \\
            R\Gamma_{(f)}(\dgrcomp{f}{M}) \arrow[r] & \dgrcomp{f}{M} \arrow[r] & {\dgrcomp{f}{M}[1/f]}.
        \end{tikzcd}
    \end{equation}
    It is sufficient to show that the right square in \eqref{local-comp} is a pullback square in \(\mcalD_{\graded{G}}(R)\) \footnote{Our argument is based on the proof of Beauville--Laszlo theorem in \cite{lurie2018Spectral}*{Lemma 7.4.1.2}}. To verify that the square is a pullback square, let \(K\) be the fiber of the morphism 
    \[
    M \to \dgrcomp{f}{M} \times_{\dgrcomp{f}{M}[1/f]} M[1/f];
    \]
    it then suffices to show that \(K = 0\).
    Since \(\dgrcomp{f}{-}\) commutes with limits (\cite{ishizuka2026Derived}*{Lemma 4.8(2)}), taking the derived gradedwise \(f\)-completion of the above morphism yields
    \[
    \dgrcomp{f}{M} \to \dgrcomp{f}{\dgrcomp{f}{M} \times_{\dgrcomp{f}{M}[1/f]} M[1/f]} \simeq \dgrcomp{f}{M} 
    \]
    because \(\dgrcomp{f}{N[1/f]} = 0\) for any \(N\). This implies \(\dgrcomp{f}{K} = 0\).
    In particular, we have $T(K,f) \defeq \lim(\cdots \xrightarrow{\times f} K) \xrightarrow{\simeq} K$ by \cite{ishizuka2026Derived}*{Definition 4.3}.
    Furthermore, we have
    \[
    M[1/f] \to  (\dgrcomp{f}{M} \times_{\dgrcomp{f}{M}[1/f]} M[1/f])[1/f] \simeq M[1/f]
    \]
    since pullback squares and pushout squares coincide in stable \(\infty\)-categories.
    Thus we have $K[1/f]\simeq 0$.
    In conclusion, we obtain
    \[
    K \simeq T(K,f) \simeq T(K,f)[1/f] \simeq  K[1/f] \simeq 0,
    \]
    where the second equivalence follows from the fact that multiplying by \(f\) induces an isomorphism on \(T(K, f)\).

    \Cref{PropertiesCompLocCohoCompLoc}: We have \(\dgrcomp{I}{R\Gamma_{(f)}(M)} \xrightarrow{\cong} \dgrcomp{I}{M}\) for any \(f \in I\), because \(\dgrcomp{I}{M[1/f]} = 0\) and we have the fiber sequence from \Cref{PropertiesCompLocCohoFiberSeq}.
    In particular, the case \(I = (f)\) is proved.
    Assume that \(I = (f_1, \ldots, f_r)\) and that the conclusion holds for the ideal generated by \(f_1, \ldots, f_{r-1}\).
    Then we have
    \begin{align*}
        \dgrcomp{I}{R\Gamma_I(M)} & \cong \dgrcomp{I}{R\Gamma_{(f_1, \dots, f_{r-1})}(\dgrcomp{I}{R\Gamma_{(f_r)}(M)})} \cong \dgrcomp{I}{R\Gamma_{(f_1, \dots, f_{r-1})}(M)} \\
        & \cong \dgrcomp{(f_r)}{\dgrcomp{(f_1, \dots, f_{r-1})}{R\Gamma_{(f_1, \dots, f_{r-1})}(M)}} \cong \dgrcomp{I}{M}
    \end{align*}
    in \(\mcalD_{\graded{G}}(R)\), where the first isomorphism follows from \Cref{PropertiesCompLocCohoComposition}.
\end{proof}

\subsection{Local cohomology of absolute graded perfectoidization}

\begin{lemma} \label{GradedPerfdLocalCoho}
    Let $P$ be a \(G\)-graded perfectoid ring with a homogeneous ideal of definition \(I\) containing \(p\).
    Take a finitely generated homogeneous ideal \(\mfrakm\) of \(P\).
    Then the canonical morphism
    \begin{equation*}
        \dgrcomp{I}{R\Gamma_{\mfrakm}(P)} \xrightarrow{\cong} R\Gamma_{\mfrakm}^{\wedge}(P)
    \end{equation*}
    is an isomorphism in \(\mcalD_{\graded{G}}(P)\), where \(R\Gamma_{\mfrakm}^{\wedge}(P)\) is defined in \Cref{GradedwiseCompletedLocCoho} with respect to the \(I\)-adic topology.
\end{lemma}

\begin{proof}
    First, note that by \Cref{GradedwiseCompletedCech}, the right-hand side can be represented by the complex
    \begin{equation*}
        0 \to \grcomp{I}{P} \to \bigoplus_{i = 1}^r \grcomp{I}{P[1/x_i]} \to \cdots \to \grcomp{I}{P[1/x_1 \cdots x_r]} \to 0
    \end{equation*}
    in \(\Ch^b(\Mod_{\graded{G}}(P))\).
    To apply \cite{ishizuka2026Derived}*{Lemma 4.12} to the \v{C}ech complex \(R\Gamma_{\mfrakm}(P)\), it suffices to show that the canonical morphism
    \begin{equation*}
        \dgrcomp{I}{P[1/\mbfx]} \to \grcomp{I}{P[1/\mbfx]}
    \end{equation*}
    is an isomorphism in \(\mcalD_{\graded{G}}(P)\) for any \(\mbfx = x_{i_1} \cdots x_{i_k}\).
    This is precisely \Cref{LocalizationGradedPerfd}(2).
\end{proof}

\begin{theorem} \label{LocCohoCommutesAbsPerfd}
    Let \(R\) be a \(G\)-graded ring with \(G = G[1/p]\), and let \(I\) be a homogeneous ideal of definition of \(R\) containing \(p\). Let \(\mfrakm\) be a finitely generated homogeneous ideal of \(R\).
    Then there exists a natural isomorphism
    \begin{equation*}
        \dgrcomp{I}{R\Gamma_{\mfrakm}(R_{\tgrpfd})} \xrightarrow{\cong} \grlim_{P \in \Perfd_{\graded{G}}^{\wedge I}} \dgrcomp{I}{R\Gamma_{\mfrakm}(P)}
    \end{equation*}
    in \(\mcalD_{\graded{G}}(R)\).
\end{theorem}

\begin{proof}
    We proceed by induction on the number \(r\) of homogeneous generators \(f_1, \dots, f_r\) of \(\mfrakm\).
    For \(r = 1\), \Cref{PropertiesCompLocCoho}\Cref{PropertiesCompLocCohoFiberSeq} yields a fiber sequence
    \begin{equation*}
        \dgrcomp{I}{R\Gamma_{(f)}(P)} \to \dgrcomp{I}{P} \to \dgrcomp{I}{P[1/f]}
    \end{equation*}
    in \(\mcalD_{\graded{G}}(R)\) for each \(I\)-adic \(G\)-graded perfectoid \(R\)-algebra \(P\).
    By \Cref{DerivedGrcompGradedPerfd} and the gradedwise \(I\)-completeness of \(P\), the canonical morphism
    \begin{equation*}
        \dgrcomp{I}{P} \xrightarrow{\cong} P
    \end{equation*}
    in \(\mcalD_{\graded{G}}(R)\) is an isomorphism.
    Therefore, the above fiber sequence becomes a fiber sequence
    \begin{equation*}
        \dgrcomp{I}{R\Gamma_{(f)}(P)} \to P \to \dgrcomp{I}{P[1/f]}
    \end{equation*}
    in \(\mcalD_{\graded{G}}(R)\).
    Taking the limit over all such \(P\), we obtain a fiber sequence
    \begin{equation*}
        \grlim_{R \to P} (\dgrcomp{I}{R\Gamma_{(f)}(P)}) \to R_{\tgrpfd} \to \grlim_{R \to P} (\dgrcomp{I}{P[1/f]})
    \end{equation*}
    in \(\mcalD_{\graded{G}}(R)\).
    Using the projection \(R_{\tgrpfd} \to P\) and \Cref{PropertiesCompLocCoho}\Cref{PropertiesCompLocCohoFiberSeq} for \(R_{\tgrpfd}\), we have the following commutative diagram in \(\mcalD_{\graded{G}}(R)\) whose horizontal sequences are fiber sequences:
    \begin{center}
        \begin{tikzcd}
            {\grlim_{R \to P} (\dgrcomp{I}{R\Gamma_{(f)}(P)})} \arrow[r]      & {R_{\tgrpfd}} \arrow[r]      & {\grlim_{R \to P} (\dgrcomp{I}{P[1/f]})}  \\
            {\dgrcomp{I}{R\Gamma_{(f)}(R_{\tgrpfd})}} \arrow[u] \arrow[r] & {R_{\tgrpfd}} \arrow[u, Rightarrow, no head] \arrow[r] & {\dgrcomp{I}{R_{\tgrpfd}[1/f]}} \arrow[u]
        \end{tikzcd}
    \end{center}
    The right vertical morphism corresponds to the isomorphism \eqref{CommutesLimitIsom} in \Cref{CommutativityLimitsPerfd}, hence the left vertical morphism is also an isomorphism.

    Assume that the desired isomorphism holds for the local cohomology with respect to \(\mfrakn \defeq (f_1, \dots, f_{r-1})\).
    Using \Cref{PropertiesCompLocCoho}\Cref{PropertiesCompLocCohoFiberSeq} and \Cref{PropertiesCompLocCohoComposition} for \(M = \dgrcomp{I}{R\Gamma_{\mfrakn}(-)}\), we have a commutative diagram
    \begin{equation} \label{DiagramInductionLocCohoLimits}
        \begin{tikzcd}[column sep=small]
            {\grlim_{R \to P} \dgrcomp{I}{R\Gamma_{\mfrakm}(P)}} \arrow[r]    & {\grlim_{R \to P} \dgrcomp{I}{R\Gamma_{\mfrakn}(P)}} \arrow[r]    & {\grlim_{R \to P} \dgrcomp{I}{\dgrcomp{I}{R\Gamma_{\mfrakn}(P)}[1/f_r]}}      \\
            {\dgrcomp{I}{R\Gamma_{\mfrakm}(R_{\tgrpfd})}} \arrow[u] \arrow[r] & {\dgrcomp{I}{R\Gamma_{\mfrakn}(R_{\tgrpfd})}} \arrow[u] \arrow[r] & {\dgrcomp{I}{\dgrcomp{I}{R\Gamma_{\mfrakn}(R_{\tgrpfd})}[1/f_r]}} \arrow[u]
        \end{tikzcd}
    \end{equation}
    in \(\mcalD_{\graded{G}}(R)\) whose horizontal sequences are fiber sequences.
    By the induction hypothesis, the middle vertical morphism is an isomorphism.
    By \Cref{PropertiesCompLocCoho}\Cref{PropertiesCompLocCohoLocalization}, the right vertical morphism can be identified with the morphism
    \begin{equation*}
        \dgrcomp{I}{R\Gamma_{\mfrakn}(R_{\tgrpfd}[1/f_r])} \to \grlim_{R \to P} \dgrcomp{I}{R\Gamma_{\mfrakn}(P[1/f_r])}.
    \end{equation*}
    By \Cref{PropertiesCompLocCoho}\Cref{PropertiesCompLocCohoCompLoc} and \Cref{PropertiesCompLocCohoStable}, and by replacing \(\dgrcomp{I}{P[1/f_r]}\) with \(\grcomp{I}{P[1/f_r]}\) via \Cref{LocalizationGradedPerfd}(2), this morphism is the derived gradedwise \(I\)-completion of
    \begin{equation*}
        R\Gamma_{\mfrakn}(\dgrcomp{I}{R_{\tgrpfd}[1/f_r]}) \to \grlim_{P \in \Perfd_{\graded{G}}^{\wedge I}} R\Gamma_{\mfrakn}(\grcomp{I}{P[1/f_r]}).
    \end{equation*}
    
    Furthermore, by \Cref{CommutativityLimitsPerfd} and the equivalence between the categories \(\set{\grcomp{I}{P[1/f_r]}}{P \in \Perfd^I_{\graded{G}}(R)}\) and \(\Perfd^I_{\graded{G}}(R[1/f_r]^{\wedge})\), this morphism translates to
    \begin{equation*}
        R\Gamma_{\mfrakn}((R[1/f_r]^{\wedge})_{\tgrpfd}) \to \grlim_{R[1/f_r]^{\wedge} \to P'} R\Gamma_{\mfrakn}(P')
    \end{equation*}
    and this is an isomorphism in \(\mcalD_{\graded{G}}(R[1/f_r]^{\wedge})\) by the induction hypothesis.
    This implies that the left vertical morphism in \eqref{DiagramInductionLocCohoLimits} is also an isomorphism, completing the proof.
\end{proof}

\begin{theorem} \label{ExactTriLocalCechCoho}
    Let \(R\) be a \(\setQ^n_{\geq 0}\)-graded adic ring with a homogeneous ideal of definition \(I\) containing $p$.
    Let \(G \subseteq \setQ^n\) be the subgroup generated by the support \(\Supp(R)\) and \(1/p\).
    Let \(R_+\) be the homogeneous ideal of \(R\) generated by all homogeneous elements whose degrees belong to \((\setQ_{>0})^n\), and assume that \(R_+\) is generated by finitely many homogeneous elements \(x_1, \dots, x_r \in R\).\footnote{Note that, in the case of \(n > 1\), this \(R_+\) is not the same as the homogeneous ideal generated by homogeneous elements whose degrees are not \(0 = (0, \dots, 0)\).}
    By \cite{ishizuka2026Algebraization}*{Construction A.1}, we can associate the projective spectrum \(\Proj(R) \defeq \Proj(R; R_+)\) to \(R\) with respect to \(R_+\).
    Then we have a fiber sequence
    \begin{equation*}
        \dgrcomp{I}{R\Gamma_{R_+}(R_{\tperfd,\graded})} \to R_{\tperfd,\graded} \to \bigoplus_{\mbfq \in G} R\Gamma(\hProj(R), \lim_{P \in\Perfd_{\graded{G}}^{\wedge I}(R)} P(\mbfq)^\Delta)
    \end{equation*}
    in \(\mcalD^{\comp{I}}_{\graded{G}}(R)\), where \(\hProj(R)\) is the \(I\)-adic completion of the \(R\)-scheme \(\Proj(R)\) and \(P(\mbfq)^{\Delta}\) is the limit \(\lim_{n \geq 1} \widetilde{P(\mbfq)/I^n}\) of the associated \(\mcalO_{\Proj(R)}\)-modules.\footnote{This \(P(\mbfq)^{\Delta}\) is isomorphic to the \(\widetilde{I}\)-adic completion of the quasi-coherent sheaf \(\widetilde{P(\mbfq)}\) on \(\Proj(R)\). See \cite{ishizuka2026Algebraization}*{Construction A.8}.}


\end{theorem}

\begin{proof}
Set $\mscrX \defeq \hProj(R)$.
Let $P \in \Perfd_{\graded{G}}^{\wedge I}(R)$.
Choose the standard affine open covering \(\mcalU \defeq \{D_+(x_i)\}_{i=1}^r\) of \(\mscrX\).
By \cite{gabber2018Foundations}*{Theorem 15.1.37 (ii)}, $R\Gamma(\mscrX, P(\mbfq)^{\Delta}/I^n)$ is computed by the \v{C}ech complex \(\cech(\mcalU, P(\mbfq)^{\Delta}/I^n)\).
Following the argument in \Cref{GradedwiseCompletedCech}, the vanishing of \(R^1\lim_{n \geq 1}\) in each cohomological degree yields the representation
\begin{align*}
    & R\Gamma(\mscrX, P(\mbfq)^{\Delta}) \cong R\lim_{n \geq 1}R\Gamma(\mscrX, \widetilde{P(\mbfq)}/I^n) \\
    & \cong (0 \to \bigoplus_{i=1}^r (\grcomp{I}{P(\mbfq)_{x_i}})_0 \to \bigoplus_{1 \leq i < j \leq r} (\grcomp{I}{P(\mbfq)_{x_ix_j}})_0 \to \cdots \to (\grcomp{I}{P(\mbfq)_{x_1 \cdots x_r}})_0 \to 0)
\end{align*}
in \(\mcalD(R)\) by the \(I\)-adic completeness of \(P(\mbfq)^{\Delta}\), where the term \(\bigoplus_{i=1}^r (\grcomp{I}{P(\mbfq)_{x_i}})_0\) sits in cohomological degree \(0\).
Comparing the representations obtained here and in \Cref{GradedwiseCompletedCech}, we obtain a canonical fiber sequence \((R\Gamma^{\wedge}_{R_+}(P))_{\mbfq} \to P_{\mbfq} \to R\Gamma(\mscrX, P(\mbfq)^{\Delta})\) in \(\mcalD(R)\).
Taking the direct sum over all \(\mbfq \in G\), we obtain the following fiber sequence in \(\mcalD_{\graded{G}}(R)\):
\begin{equation} \label{FiberSequenceGrPerfd}
    R\Gamma^{\wedge}_{R_+}(P) \to P \to \bigoplus_{\mbfq \in G} R\Gamma(\hProj(R), P(\mbfq)^{\Delta}).
\end{equation}
Taking the limit in $\mcalD_{\graded{G}}^{\comp{I}}(R)$ over all $P \in \Perfd_{\graded{G}}^{\wedge I}(R)$ and applying \Cref{GradedPerfdLocalCoho} and \Cref{LocCohoCommutesAbsPerfd}, we obtain the desired fiber sequence in $\mcalD^{\comp{I}}_{\graded{G}}(R)$:
\begin{equation*}
        \dgrcomp{I}{R\Gamma_{R_+}(R_{\tperfd,\graded})} \to R_{\tperfd,\graded} \to \bigoplus_{\mbfq \in G} R\Gamma(\hProj(R), \lim_{P \in \Perfd_{\graded{G}}^{\wedge I}(R)} P(\mbfq)^{\Delta}),
    \end{equation*}
as desired.
\end{proof}

\section{Local-global correspondence} \label{SectionLocalGlobal}

In this section, we will show the local and global correspondence for lim-perfectoid splitting.
We work on the setting below.

\begin{notation} \label{SettingFibersequenceRecall}
    Recall that \(p\) be a prime number.
\begin{itemize}
\item Let $(R_0,\m_0, k)$ be a \(p\)-torsion-free discrete valuation ring whose residue characteristic is \(p\).
We fix a normalized dualizing complex \(\omega_{R_0}^{\bullet}\).
\item Let $X$ be a flat projective scheme over $R_0$ such that $H^0(X,\cO_X)=R_0$.
We set $d\defeq \dim X$ and assume \(d \geq 2\).
\item We define $\omega_X^{\bullet} \defeq (X \to \Spec R_0)^{!}\omega_{R_0}^{\bullet}$ and $\omega_X\defeq \mcalH^{-d}(\omega_X^{\bullet})$.
\item Let \(L\) be an ample line bundle on \(X\), and let \(L^{-1}\) denote its dual, equipped with a fixed isomorphism \(L \otimes_{\mcalO_X} L^{-1} \xrightarrow{\cong} \mcalO_X\).
Let
\begin{equation*}
    R \defeq \bigoplus_{n \geq 0} H^0(X, L^{\otimes n})
\end{equation*}
be the section ring, with the homogeneous ideal \(R_+ \defeq \bigoplus_{n > 0} H^0(X, L^{\otimes n})\), and let \(\Proj(R)\) be its projective spectrum.
We can identify \(X\) with \(\Proj(R)\) and \(L\) on \(X\) with \(\mcalO_{\Proj(R)}(1)\) on \(\Proj(R)\) (\cite{ishizuka2026Algebraization}*{Lemma 5.2}).
\item Let $\mscrX$ be the $p$-adic formal completion of $X$.
\item Let \(\omega_{\mscrX}\) be the \(p\)-adic completion of \(\omega_X\), and let \(\omega_{\mscrX}^{\bullet}\) be the derived \(p\)-completion \((\omega_{X}^{\bullet})^{\wedge} \in \mcalD_{\qcoh}(\mscrX)\) of \(\omega_X^{\bullet}\) defined in \Cref{DerivedCompletionFormalSchemes}.
\end{itemize}
Note that $\omega_X$ is a dualizing sheaf since $\omega_X^{\bullet}$ is a normalized dualizing complex by \cite{KTTWYY1}*{Proposition~2.5}.
Moreover, \(\mcalH^{-d}(\omega_{\mscrX}^{\bullet})\) is isomorphic to \(\omega_{\mscrX}\) by a variant of the formal function theorem (e.g., \cite{gortz2023Algebraic}*{Proposition 24.33}).
Furthermore, the symbol \(R\Gamma_{(p)}(\mscrX, -)\) denotes \(R\Gamma_{(p)}(R\Gamma(\mscrX, -))\).
\end{notation}

\begin{remark}\label{rmk:Gorensteiness-section-ring}
In \Cref{SettingFibersequenceRecall}, we further assume that $\cO_X \simeq \omega_X$, or $L=\omega_X$, or $L=\omega_X^{-1}$.
Then $R_{(\m_0,R_+)}$ is Gorenstein if and only if $R_{(\m_0,R_+)}$ is Cohen--Macaulay by \Cref{dualizing-module-affine-cone}.    
\end{remark}


\subsection{Local-to-global correspondence}

Using the algebraization result of \(\mcalO_{\mscrX, \perfd}\), we can show the following:

\begin{proposition} \label{LocalGlobalPure}
    Keep the setting of \Cref{SettingFibersequenceRecall}.
    If the \(p\)-adic completion \(R^{\wedge p}\) of \(R\) is lim-perfectoid pure (resp., perfectoid pure), then \(X\) is lim-perfectoid split (resp., perfectoid split).
\end{proposition}

\begin{proof}
    First, we assume $R^{\wedge p}$ is lim-perfectoid pure.
    By \Cref{EquivGradedPerfdPure}, the canonical morphism \(R \to R_{\grpfd}\) is also ind-split in \(\mcalD_{\graded{\setZ[1/p]}}(R)\).
    Taking the associated sheaf, we have an ind-split morphism
    \begin{equation*}
        \mcalO_X \cong \widetilde{R} \to \widetilde{R_{\grpfd}} = \mcalO_{X, \perfd}
    \end{equation*}
    in \(\mcalD_{\qcoh}(X)\) by \Cref{AlgebraizableOXperfd}.
    Therefore, \(X\) is lim-perfectoid split by \Cref{ind-split-dual}.

Next, we assume $R^{\wedge p}$ is perfectoid pure.
Hence, by \cite{ishizuka2025Graded}*{Proposition 4.24}, there exists a $\setZ[1/p]$-graded perfectoid $R$-algebra $R_{\infty}$ such that the induced homomorphism $R \to R_{\infty}$ is pure.
The morphism \(R \to R_{\infty}\) induces an affine morphism
\[
\Proj(R_{\infty}) \supseteq Y \defeq \bigcup_{f \in R_+} D_+(\varphi(f)) \xrightarrow{\pi} \Proj(R) = X.
\]
Since $R_{\infty}$ is a graded perfectoid ring, the $p$-adic completion of $\Proj(R_{\infty})$ is a perfectoid formal scheme by \cite{ishizuka2026Algebraization}*{Proposition B.18}.
Therefore, the open subscheme $\wt{Y}$ is a perfectoid formal scheme.
By construction, we have
\[
(\pi_*\mcalO_{Y})^{\wedge} \cong \widehat{\pi}_*\cO_{\widehat{Y}} \cong R_{\infty}^{\Delta},
\]
where $\widehat{\pi}$ denotes the $p$-adic formal completion of the morphism $\pi \colon Y \to X$.
By the above argument, we have the ind-split homomorphism
\[
\cO_X \cong R^{\Delta} \to R_{\infty}^{\Delta} \cong \wt{\pi}_*\mcalO_{\wt{Y}}.
\]
Thus, the assertion follows from \Cref{ind-split-dual}.
\end{proof}

The above proposition requires the lim-perfectoid purity on \(R^{\wedge p}\).
In what follows, we will relax this assumption to the lim-perfectoid injectivity (\Cref{local-to-global-PS}).

Using the twist of the structure sheaf of \(\mscrX\), we can define the twist of \(\mcalO_{\mscrX, \perfd}\).

\begin{definition}[{\cite{ishizuka2026Algebraization}*{Definition 5.8 and Corollary 5.9}}] \label{DefTwistedPerfdProj}
    In the setting of \Cref{SettingFibersequenceRecall}, we define the twisted absolute perfectoidization
    \begin{equation*}
        \mcalO_{\mscrX, \perfd}(m) \defeq \mcalO_{\mscrX, \perfd} \otimes^L_{\mcalO_{\mscrX}} \mcalO_{\mscrX}(m) \in \mcalD(\mscrX_{\Zar}, \mcalO_{\mscrX})
    \end{equation*}
    for every \(m \in \setZ\).
    By the cited reference, there is a canonical isomorphism
    \begin{equation*}
        \mcalO_{\mscrX, \perfd}(m) \xrightarrow{\cong} \lim_{P \in \Perfd^{\wedge p}_{\graded{\setZ[1/p]}}(R)} P(m)^{\Delta}
    \end{equation*}
    in \(\mcalD(\mscrX, \mcalO_{\mscrX})\) for every \(m \in \setZ\).
\end{definition}


Since \(R\) is only \(\setZ_{\geq 0}\)-graded, we cannot define the twist \(\mcalO_{\mscrX, \perfd}(q)\) for \(q \in \setZ[1/p] \setminus \setZ\) by using the twist on \(\mscrX\) as in \Cref{DefTwistedPerfdProj}.
However, the isomorphism established above allows us to define such a twist as follows:

\begin{definition} \label{RationalTwistAbsPerfd}
    Keep the setting of \Cref{SettingFibersequenceRecall}.
    For any \(q \in \setZ[1/p]\), we define an object in \(\mcalD(\mscrX, \mcalO_{\mscrX})\) by
    \begin{equation*}
        \mcalO_{\mscrX, \perfd}(q) \defeq \lim_{P \in \Perfd^{\wedge p}_{\graded{\setZ[1/p]}}(R)} P(q)^{\Delta},
    \end{equation*}
    which we call the \emph{\(q\)-twist of \(\mcalO_{\mscrX, \perfd}\)} (or the \emph{twist by \(q\)}), denoted by \(\mcalO_{\mscrX, \perfd}(q)\).
\end{definition}



\begin{remark} \label{RemarkRationalTwist}
    Regarding the definition of \(\mcalO_{\mscrX, \perfd}(q)\), the following properties hold:
    \begin{enumerate}
        \item For \(q = m \in \setZ\), this definition coincides with the definition of \(\mcalO_{\mscrX, \perfd}(m)\) in \Cref{DefTwistedPerfdProj}.
        \item For any \(q \in \setZ[1/p]\), the absolute graded perfectoidization \(R_{\grpfd}\) has the grade shifting \(R_{\grpfd}(q)\) in \(\mcalD_{\graded{\setZ[1/p]}}(R)\). By using this, we can take an object
        \begin{equation*}
            \mcalO_{X, \perfd}(q) \defeq \widetilde{R_{\grpfd}(q)}^{\wedge} \in \mcalD_{\qcoh}(X)
        \end{equation*}
        by using the derived \(p\)-completion and the associated object by the same construction in \Cref{AlgebraizableOXperfd}. As in the proof of \Cref{AlgebraizableOXperfd}, we have isomorphisms
        \begin{align*}
            & R\Gamma(D_+(f), \mcalO_{X, \perfd}(q)^{\wedge}) \cong (R[1/f]_{\grpfd})_q = \lim_{Q' \in \Perfd^{\wedge p}_{\graded{\setZ[1/p]}}} Q'_q \\
            & \cong \lim_{P \in \Perfd^{\wedge p}_{\graded{\setZ[1/p]}}} (\grcomp{p}{P[1/f]})_q \cong R\Gamma(D_+(f), \mcalO_{\mscrX, \perfd}(q))
        \end{align*}
        for any \(f \in R_+\). Therefore, we have another representation of \(\mcalO_{\mscrX, \perfd}(q)\):
        \begin{equation*}
            \mcalO_{\mscrX, \perfd}(q) \cong \mcalO_{X, \perfd}(q)^{\wedge}
        \end{equation*}
        in \(\mcalD_{\qcoh}(\mscrX)\). Namely, we can algebraize the \(q\)-twist of \(\mcalO_{\mscrX, \perfd}\) for any \(q \in \setZ[1/p]\).
        \item By (1), we obtain a natural split inclusion
        \begin{equation*}
            \bigoplus_{m \in \setZ} R\Gamma(\mscrX, \mcalO_{\mscrX, \perfd}(m)) \hookrightarrow \bigoplus_{q \in \setZ[1/p]} R\Gamma(\mscrX, \mcalO_{\mscrX, \perfd}(q))
        \end{equation*}
        in \(\CAlg(\mcalD_{\graded{\setZ[1/p]}}(R))\).
    \end{enumerate}
\end{remark}

Using these twists, we obtain the following fiber sequence.

\begin{proposition} \label{GradedAbsPerfdSheafCoho}
    In the setting of \Cref{SettingFibersequenceRecall}, there exists a fiber sequence
    \begin{equation*}
        \dgrcomp{(p)}{R\Gamma_{R_+}(R_{\grpfd})} \to R_{\grpfd} \to \bigoplus_{q \in \setZ[1/p]} R\Gamma(\mscrX, \mcalO_{\mscrX, \perfd}(q))
    \end{equation*}
    in \(\mcalD_{\graded{\setZ[1/p]}}^{\comp{p}}(R)\).
    In particular, \(R\Gamma_{(p)}(R_{\perfd})\) and \(R\Gamma_{(p, R_+)}(R_{\perfd})\) naturally inherit \(\setZ[1/p]\)-graded structures, and with respect to these structures, we obtain a fiber sequence
    \begin{equation*}
        R\Gamma_{(p, R_+)}(R_{\perfd}) \to R\Gamma_{(p)}(R_{\perfd}) \to \bigoplus_{q \in \setZ[1/p]} R\Gamma_{(p)}(\mscrX, \mcalO_{\mscrX, \perfd}(q))
    \end{equation*}
    in \(\mcalD_{\graded{\setZ[1/p]}}(R)\).
\end{proposition}

\begin{proof}
    The first fiber sequence follows directly from \Cref{ExactTriLocalCechCoho} and \Cref{DefTwistedPerfdProj}.
    Applying the local cohomology functor \(R\Gamma_{(p)}(-)\) yields the second fiber sequence, where we use \Cref{PropertiesCompLocCoho}\Cref{PropertiesCompLocCohoStable}, \Cref{DefGradedAbsPerfd}, and \Cref{GradedAbsPerfdProperties}\Cref{CompletionOfAbsoluteGradedPerfd}.
    We note that
    \begin{align*}
    R\Gamma_{(p)}\left(\bigoplus_{q \in \setZ[1/p]} R\Gamma(\mscrX, \mcalO_{\mscrX, \perfd}(q)) \right)
    &\simeq R\Gamma_{(p)}\left(\bigoplus_{q \in \setZ[1/p]} R\Gamma(\mscrX, \mcalO_{\mscrX, \perfd}(q)) \right) \\
    &\simeq \bigoplus_{q \in \setZ[1/p]} R\Gamma_{(p)}(\mscrX, \mcalO_{\mscrX, \perfd}(q)). 
    \end{align*}
\end{proof}

\begin{theorem}\label{LocalToGlobal}
    Keep the setting of \Cref{SettingFibersequenceRecall}.
    Let \(i \geq 2\) be an integer.
    Then we have a commutative diagram
    \begin{equation*}
        \begin{tikzcd}
            {\bigoplus_{q \in \setZ[1/p]} H^i_{(p)}(\mscrX, \mcalO_{\mscrX, \perfd}(q))} \arrow[r]       & {H^{i+1}_{(p, R_+)}(R_{\perfd})}   \\
            {\bigoplus_{m \in \setZ} H^i_{(p)}(\mscrX, \mcalO_{\mscrX, \perfd}(m))} \arrow[u, hook]      &                                    \\
            {\bigoplus_{m \in \setZ} H^i_{(p)}(\mscrX, \mcalO_{\mscrX}(m))} \arrow[r, "\cong"] \arrow[u] & {H^{i+1}_{(p, R_+)}(R)} \arrow[uu]
        \end{tikzcd}
    \end{equation*}
    of \(\setZ[1/p]\)-graded \(R\)-modules.
    In particular, if \(H^{i+1}_{(p, R_+)}(R) \to H^{i+1}_{(p, R_+)}(R_{\perfd})\) is injective, then so is \(H^i_{(p)}(\mscrX, \mcalO_{\mscrX}(m)) \to H^i_{(p)}(\mscrX, \mcalO_{\mscrX, \perfd}(m))\) for any \(m \in \setZ\).
\end{theorem}

\begin{proof}

    By the construction of the fiber sequences in \Cref{GradedAbsPerfdSheafCoho} and \Cref{ExactTriLocalCechCoho}, we have the following commutative diagram
    \begin{equation*}
        \begin{tikzcd}
            {R\Gamma_{(p, R_+)}(R_{\perfd})} \arrow[r]   & R\Gamma_{(p)}(R_{\perfd}) \arrow[r]   & {\bigoplus_{q \in \setZ[1/p]} R\Gamma_{(p)}(\mscrX, \mcalO_{\mscrX, \perfd}(q))}            \\
             &                                   & {\bigoplus_{m \in \setZ} R\Gamma_{(p)}(\mscrX, \mcalO_{\mscrX, \perfd}(m))} \arrow[u, hook] \\
            {R\Gamma_{(p, R_+)}(R)} \arrow[r] \arrow[uu] & R\Gamma_{(p)}(R) \arrow[r] \arrow[uu] & {\bigoplus_{m \in \setZ} R\Gamma_{(p)}(\mscrX, \mcalO_{\mscrX}(m))} \arrow[u]              
        \end{tikzcd}
    \end{equation*}
    in \(\mcalD_{\graded{\setZ[1/p]}}(R)\) whose horizontal sequences are fiber sequences, where the right morphism is the composition of the morphism induced from \(\mcalO_{\mscrX} \to \mcalO_{\mscrX, \perfd}\) and the inclusion to \(\bigoplus_{q \in \setZ[1/p]} R\Gamma(\mscrX, \mcalO_{\mscrX, \perfd}(q))\) given by \Cref{RemarkRationalTwist}(4).

    The local cohomology \(H^i_{(p)}(-)\) vanishes for \(i > 1\).
    Taking the associated long exact sequence in cohomology, we obtain the following commutative diagram
    \begin{equation*}
        \begin{tikzcd}
            H^i_{(p)}(R_{\perfd}) \arrow[r]         & {\bigoplus_{q \in \setZ[1/p]} H^i_{(p)}(\mscrX, \mcalO_{\mscrX, \perfd}(q))} \arrow[r]       & {H^{i+1}_{(p, R_+)}(R_{\perfd})}   \\
            & {\bigoplus_{m \in \setZ} H^i_{(p)}(\mscrX, \mcalO_{\mscrX, \perfd}(m))} \arrow[u, hook]      &                                    \\
            0 = H^i_{(p)}(R) \arrow[uu] \arrow[r] & {\bigoplus_{m \in \setZ} H^i_{(p)}(\mscrX, \mcalO_{\mscrX}(m))} \arrow[r, "\cong"] \arrow[u] & {H^{i+1}_{(p, R_+)}(R)} \arrow[uu]
        \end{tikzcd}
    \end{equation*}
    of \(\setZ[1/p]\)-graded \(R\)-modules for \(i > 1\).
    \qedhere
\end{proof}

\begin{corollary}[Local-to-global correspondence]\label{local-to-global-PS}
Keep the setting of \Cref{SettingFibersequenceRecall}.
Assume that $X$ is Cohen--Macaulay, and that either $\omega_X \simeq \cO_X$, $L=\omega_X^{-1}$, or $L=\omega_X$.
Then the following assertions hold:
\begin{enumerate}
    \item If $H^{d+1}_{(p,R_+)}(R) \to H^{d+1}_{(p,R_+)}(R_{\perfd})$ is injective, then $X$ lim-perfectoid split.
    \item If $H^{d+1}_{(p,R_+)}(R) \to H^{d+1}_{(p,R_+)}(P)$ is injective for some perfectoid \(R\)-algebra \(P\), then \(X\) is perfectoid split.
\end{enumerate}

\end{corollary}

\begin{proof}
(1): We have an injection
\[
H^{d+1}_{(p,R_+)}(R) \to H^{d+1}_{(p,R_+)}(R_{\perfd}).
\]
By \Cref{LocalToGlobal} and the assumption on $\omega_X$, we obtain an injection
\[
H^d_{(p)}(\mscrX, \mcalO_{\mscrX} \otimes \omega_{\mscrX}) \to H^d_{(p)}(\mscrX, \mcalO_{\mscrX, \perfd}\otimes^L \omega_{\mscrX}).
\]
Therefore, by the Cohen--Macaulayness of $X$ and \Cref{ind-split-dual}, we conclude that $X$ is lim-perfectoid split.

(2): By \cite{ishizuka2025Graded}*{Proposition 4.24}, there exists a $\setZ[1/p]$-graded perfectoid $R$-algebra $R_{\infty}$ such that the induced homomorphism
\[
H^{d+1}_{(p,R_+)}(R) \to H^{d+1}_{(p,R_+)}(R_{\infty})
\]
is injective.
As in the proof of \Cref{local-to-global-PS}, we can take an affine morphism \(\pi \colon \mscrY \to \mscrX\) from a perfectoid formal scheme \(\mscrY\) such that
\[
\pi_*\cO_{\mscrY} \cong R_{\infty}^{\Delta}
\]
as $\cO_{\mscrX}$-algebras.
Moreover, by \eqref{FiberSequenceGrPerfd}, we have a fiber sequence
\[
R\Gamma_{(p,R_+)}(R_{\infty}) \to R\Gamma_{(p)}(R_{\infty}) \to \bigoplus_{q \in \Z[1/p]} R\Gamma_{(p)}(\mscrX,R_{\infty}^{\Delta}(q)).
\]
Hence, by the proof of \Cref{LocalToGlobal} and the assumption on $\omega_X$, we obtain an injection
\[
H^d_{(p)}(\mscrX,\cO_{\mscrX}\otimes \omega_{\mscrX}) \to H^d_{(p)}(\mscrX,R_{\infty}^{\Delta} \otimes \omega_{\mscrX}) \cong H^d_{(p)}(\mscrX,\pi_*\cO_{\mscrY} \otimes \omega_{\mscrX}).
\]
Thus, the assertion follows from \Cref{ind-split-dual}.
\end{proof}

\subsection{Global-to-local correspondence}

In this subsection, we establish a global-to-local correspondence.

\begin{lemma} \label{VanishingPerfd}
    Let \((V, \mfrakm, k)\) be a discrete valuation ring whose maximal ideal \(\mfrakm\) contains \(p\).
    Then \(R\Gamma_{(p)}(V_{\perfd})\) in \(\mcalD(V)\) is concentrated in (cohomological) degrees \([1, 2]\).
\end{lemma}

\begin{proof}
    Since the absolute perfectoidization \(V_{\perfd}\) is isomorphic to \((\widehat{V})_{\perfd}\) for the completion of \(V\), we may assume that \(V\) is complete.
    Then there exists a complete discrete valuation ring \((W, \mfrakn, k_{\perf})\) and a faithfully flat extension \(V \hookrightarrow W\).
    Since the residue field \(k_{\perf}\) is perfect, there exists a perfectoid extension \(C\) of \(K \defeq \Frac(W)\) which is the \(p\)-adic completion of a totally ramified Galois extension of \(K\) with Galois group \(\Gamma = \setZ_p\).
    By \cite{bhatt2025Aspects}*{Proposition 4.5.1 and Example 4.5.2} (or \cite{bhatt2024Perfectoid}*{Proposition 3.17} if \(V\) is \(\setZ_p\)), we have a fiber sequence
    \begin{equation*}
        V_{\perfd} \to \fib(\mcalO_C \xrightarrow{\gamma -1} \mcalO_C) \to k_{\perf}[-1]
    \end{equation*}
    in \(\mcalD(V)\).
    Taking \(R\Gamma_{(p)}(-)\), we have a fiber sequence
    \begin{equation*}
        R\Gamma_{(p)}(V_{\perfd}) \to \fib(R\Gamma_{(p)}(\mcalO_C) \xrightarrow{\gamma - 1} R\Gamma_{(p)}(\mcalO_C)) \to R\Gamma_{(p)}(k_{\perf})[-1],
    \end{equation*}
    which becomes
    \begin{equation*}
        R\Gamma_{(p)}(V_{\perfd}) \to \fib(H^1_{(p)}(\mcalO_C) \xrightarrow{\gamma -1} H^1_{(p)}(\mcalO_C))[-1] \to k_{\perf}[-1]
    \end{equation*}
    in \(\mcalD(V)\).
    Taking the associated long exact sequence, we conclude \(H^i_{(p)}(V_{\perfd}) = 0\) for all \(i \leq 0\) and \(i \geq 3\).
\end{proof}

\begin{lemma}[cf.~\cite{bhatt2025Aspects}*{Proposition 11.2.8}] \label{GradedZeroPartInjective}
    Keep the setting of \Cref{SettingFibersequenceRecall}.
    For any \(m \in \setZ[1/p]_{\leq 0}\), consider the degree \(m\) part of the upper horizontal morphism appearing in \Cref{LocalToGlobal}:
    \begin{equation*}
        H^i_{(p)}(\mscrX, \mcalO_{\mscrX, \perfd}(m)) \to H^{i+1}_{(p, R_+)}(R_{\perfd})_m.
    \end{equation*}
    Then the following hold:
    \begin{enumerate}
        \item If \(m < 0\), this morphism is injective for any \(i \in \setZ\).
        \item If \(m = 0\), this morphism is injective for any \(i \in \setZ \setminus \{1, 2\}\).
    \end{enumerate}
\end{lemma}

\begin{proof}
    Using the second fiber sequence of \Cref{GradedAbsPerfdSheafCoho}, we obtain a long exact sequence of \(R\)-modules:
    \begin{equation*}
        H^i_{(p, R_+)}(R_{\grpfd}) \to H^i_{(p)}(R_{\grpfd}) \to \bigoplus_{q \in \setZ[1/p]} H^i_{(p)}(\mscrX, \mcalO_{\mscrX, \perfd}(q)) \to H^{i+1}_{(p, R_+)}(R_{\perfd}).
    \end{equation*}
    To establish the desired injectivity, it suffices to show the vanishing
    \begin{equation}
        H^i_{(p)}(R_{\grpfd})_m = 0
    \end{equation}
    for \(i \in \setZ\) if \(m < 0\) and for \(i \geq 3\) if \(m = 0\).
Then we have the isomorphism
    \begin{equation} \label{GradedZeroPartIsom}
        H^i_{(p)}(R_{\grpfd,m}) \xrightarrow{\sim} H^i_{(p)}(R_{\grpfd})_m.
    \end{equation}
    If $m<0$, then we have $R_{\grpfd,m}=0$ by \Cref{GradedAbsPerfdProperties}\Cref{RemarkGradedPerfd}, thus we obtain the desired vanishing.

    Using \Cref{GradedAbsPerfdProperties}\Cref{GradedPerfdZeroPart}, in the case of \(m = 0\), it remains to show the vanishing of \(H^i_{(p)}(R_{0, \perfd})\) for \(i \neq 1, 2\).
    Since \(R_0\) is a discrete valuation ring with perfect residue field, we can apply \Cref{VanishingPerfd}.
    So we are done.
    \qedhere

\end{proof}

First, we establish the global-to-local principle for the Calabi--Yau and Fano cases.

\begin{theorem} \label{GlobalToLocalCYFano}
    Keep the setting of \Cref{SettingFibersequenceRecall}.
    We  assume that $X$ is lim-perfectoid split.
    Assume either of the following:
    \begin{enumerate}
        \item $\omega_{X}$ is trivial with \(d \geq 3\), or
        \item $\omega_{X}$ is anti-ample and $\omega_{X}^{-1}= L$.
    \end{enumerate}
    Then the morphism
    \begin{equation*}
        H^{d+1}_{(p, R_+)}(R) \to H^{d+1}_{(p, R_+)}(R_{\perfd})
    \end{equation*}
    is injective.
    In particular, if $R_{(\m_0, R_+)}$ is Cohen--Macaulay, then $R^{\wedge p}$ is lim-perfectoid pure.
\end{theorem}

\begin{proof}
    Since \(X\) is lim-perfectoid split, we have an injection
    \begin{equation*}
        H^d_{(p)}(\mscrX, \omega_{\mscrX}) \to H^d_{(p)}(\mscrX, \mcalO_{\mscrX, \perfd} \otimes^L_{\mcalO_{\mscrX}} \omega_{\mscrX}).
    \end{equation*}

    By considering the long exact sequence analogous to that in the proof of \Cref{GradedZeroPartInjective} applied to \(R\), we obtain a commutative diagram
    \begin{equation} \label{GlobalToLocalCYFanoDiagram}
        \begin{tikzcd}
            {H^d_{(p)}(\mscrX, \mcalO_{\mscrX, \perfd} \otimes^L_{\mcalO_{\mscrX}} \omega_{\mscrX})} \arrow[r, hook]            & {H^{d+1}_{(p, R_+)}(R_{\perfd})_i} \arrow[r, hook]  & {H^{d+1}_{(p, R_+)}(R_{\perfd})}  \\
            {H^d_{(p)}(\mscrX, \omega_{\mscrX})} \arrow[r, "\cong"] \arrow[u, hook] & {H^{d+1}_{(p, R_+)}(R)_i} \arrow[r, hook] \arrow[u] & {H^{d+1}_{(p, R_+)}(R)} \arrow[u]
        \end{tikzcd}
    \end{equation}
    of \(R_0\)-modules, where $i=0$ if $\omega_{X}$ is trivial and $i=-1$ if $\omega_{X}^{-1}= L (= \mcalO_{X}(1))$.

    By \Cref{GradedZeroPartInjective}, the assumption (1) or (2) implies the injectivity of the upper left morphism.
    Since \(d > 1\), the lower left morphism is an isomorphism.
    Therefore, the middle vertical morphism is injective.
    Because we have $\omega_R \simeq R(i)$, the degree $i$ part \(H^{d+1}_{(p, R_+)}(R)_i\) contains the socle of \(H^{d+1}_{(p, R_+)}(R)\) by \Cref{socle}. This yields the desired injection.

    Furthermore, the final assertion follows from \Cref{rmk:Gorensteiness-section-ring} and \Cref{EquivGradedPerfdPure}.
\end{proof}

\begin{theorem}\label{Global-local-CY}
Keep the setting of \Cref{SettingFibersequenceRecall}.
Assume that $X$ is Cohen--Macaulay and that $\omega_X \simeq \cO_X$.
\begin{enumerate}
    \item If \(H^{d+1}_{(p, R_+)}(R) \to H^{d+1}_{(p, R_+)}(R_{\perfd})\) is injective, then $X$ is lim-perfectoid split.
    \item If $\dim X \geq 3$,
    then the converse of \textrm{(1)} also holds.
\end{enumerate}
In particular, if \(R_{(\mfrakm_0, R_+)}\) is Cohen--Macaulay, the injectivity is equivalent to the lim-perfectoid purity of \(R^{\wedge p}\).
\end{theorem}

\begin{proof}
The result follows from \Cref{local-to-global-PS}, \Cref{GlobalToLocalCYFano}, and \Cref{CompareGLPPGLPI}.
\end{proof}

\begin{corollary}\label{every-section-ring}
Let $k$ be a field of characteristic $p$.
Let $X \subseteq \P^N_{C_k}$ be a smooth hypersurface of degree $N+1$ over $C_k$, where $C_k$ is the Cohen ring of $k$.
If $N \geq 3$ and $X$ is lim-perfectoid split, then the $p$-adic completion of every section ring of $X$ is lim-perfectoid pure.
\end{corollary}

\begin{proof}
First, we note that the Kodaira-type vanishing theorem holds for
$\var{X}=X_{p=0}$ by \cite{Mukai}*{Proposition~3.4} and the Grothendieck--Lefschetz theorem.
Thus, for any anti-ample line bundle $M$ on $X$, it follows from the exact sequence
\[
0 \to M \xrightarrow{\cdot p} M \to M|_{\var{X}} \to 0
\]
that the homomorphism
\[
H^i_{(p)}(X,M) \xrightarrow{\cdot p} H^i_{(p)}(X,M)
\]
is injective for $i<N=\dim X$.
In particular, we have
\[
H^i_{(p)}(X,M)=0 \quad \text{for $i<\dim X$}.
\]

Furthermore, by the exact sequence
\[
0 \to \cO_{\P^N}(-(N+1)) \to \cO_{\P^N} \to \cO_X \to 0,
\]
we obtain
\[
H^i(X,\cO_X)=0 \quad \text{for $1 \leq i \leq \dim X-1$}.
\]

Therefore, every section ring of $X$ is Cohen--Macaulay by \Cref{CYSectionRingCM}.
By \Cref{Global-local-CY}, the $p$-adic completion of every section ring of $X$ is lim-perfectoid pure, as desired.
\end{proof}

\begin{theorem}\label{Global-local-Fano}
Keep the setting of \Cref{SettingFibersequenceRecall}.
Assume that $X$ is Cohen--Macaulay and $\omega_X \simeq L^{-1}$.
Then the canonical morphism \(H^{d+1}_{(p, R_+)}(R) \to H^{d+1}_{(p, R_+)}(R_{\perfd})\) is injective if and only if \(X\) is lim-perfectoid split.

In particular, if \(R_{(\mfrakm_0, R_+)}\) is Cohen--Macaulay, the former condition is equivalent to the lim-perfectoid purity of \(R^{\wedge p}\).
\end{theorem}

\begin{proof}
The result also follows from \Cref{local-to-global-PS}, \Cref{GlobalToLocalCYFano}, and \Cref{CompareGLPPGLPI}.
\end{proof}


\begin{corollary} \label{CICase}
Keep the setting of \Cref{SettingFibersequenceRecall}. Assume that the section ring \(R\) is complete intersection.
We assume either
\begin{enumerate}
    \item $\omega_{X}$ is trivial with \(d \geq 3\), or
    \item $\omega_{X}$ is anti-ample and $\omega_{X}^{-1}= L$.
\end{enumerate}
Then \(X\) is lim-perfectoid split if and only if it is perfectoid split.
\end{corollary}

\begin{proof}
    The ``if'' direction follows from \Cref{GPPtoGLPP}.
    We show the ``only if'' direction.
    By \Cref{GlobalToLocalCYFano} and \cite{bhatt2024Perfectoid}*{Corollary 4.25}, the \(p\)-adic completion \(R^{\wedge p}\) is perfectoid pure since \(R\) is complete intersection.
    Then \Cref{local-to-global-PS} shows that \(X\) is perfectoid split.    
\end{proof}

\subsection{Examples}


\begin{example}\label{Du-Val}
Let $\var{X}$ be a Du Val del Pezzo surface such that $\var{X}$ is a hypersurface in a well-formed weighted projective space over an algebraically closed field $k$ of characteristic $p$.
Then every $W(k)$-lift $X$ of $\var{X}$ is perfectoid split.
Indeed, by the proof of \cite{onuki2025QuasiFsplitting}*{Theorem~4.3}, the section ring
\[
R:=\bigoplus_{n \geq 0} H^0(X,(\omega_X^{-1})^{\otimes n})
\]
is quasi-$F$-split.
Since $\var{X}$ is hypersurface, so is $R$.
By \cite{yoshikawa2025Computation}*{Theorem~A} and \cite{ishizuka2025Graded}*{Theorem~E}, the $p$-adic completion $R^{\wedge p}$ of $R$ is perfectoid pure.
Therefore, by \Cref{LocalGlobalPure}, the lift $X$ is perfectoid split.
\end{example}

\begin{example}[{cf.~\cite{bhatt2024Perfectoid}*{Example 7.7}}] \label{ExampleWeaklyOrdinary}
    Let \(\overline{X}\) be a weakly ordinary smooth projective variety over a perfect field \(k\) of characteristic \(p\) such that \(\Omega^1_{\overline{X}/k}\) is trivial.
    Let \(X\) be the canonical \(W(k)\)-lift of \(\overline{X}\).
    If \(\omega_X\) is trivial, then \(X\) is lim-perfectoid split.
    For example, the canonical lift of an ordinary abelian variety is perfectoid split.

    It follows from \Cref{local-to-global-PS} and the fact that the \(p\)-adic completion of the section ring of \(X\) with respect to an ample line bundle is perfectoid injective by \cite{bhatt2024Perfectoid}*{Example 7.7}.

\end{example}

\begin{example}\label{elliptic-curve}
Let \(k\) be a perfect field of characteristic \(p>0\), and let \(\var{X}\) be an elliptic curve over \(k\).
Then every \(W(k)\)-lift \(X\) of \(\var{X}\) is perfectoid split.

Indeed, since \(\var{X}\) is a plane cubic, there exists a homogeneous polynomial \(\var{f} \in k[x,y,z]\) of degree \(3\) such that
\[
\var{X}=(\var{f}=0)\subseteq \P^2_k.
\]
Since \(\var{X}\) is quasi-\(F\)-split, so is
\[
\var{S}:=k[x,y,z]/(\var{f}).
\]
As \(X\) is a \(W(k)\)-lift of \(\var{X}\), there exists a homogeneous lift \(f \in W(k)[x,y,z]\) of \(\var{f}\) of degree \(3\) such that
\[
X=(f=0)\subseteq \P^2_{W(k)}.
\]
Hence
\[
S:=W(k)[x,y,z]/(f)
\]
is a section ring of \(X\). Since \(\var{S}=S/pS\) is quasi-\(F\)-split, it follows from \cite{yoshikawa2025Computation}*{Theorem~A} that the \(p\)-adic completion \(S^{\wedge p}\) is perfectoid pure. Therefore \(X\) is perfectoid split by \Cref{Global-local-CY}.
\end{example}

\begin{example}\label{CY}
Let $k$ be an algebraically closed field of characteristic $p$.
\begin{enumerate}
\item Let $X$ be a smooth Calabi--Yau hypersurface of degree $d$ in a projective space over $W(k)$.
If $p$ is larger than the relative dimension of $X$ and $p \nmid d$, then $X$ is perfectoid split by \cite{Yoshikawa26}*{Theorem~A}.
\item Let $\var{X}$ be a hypersurface Calabi--Yau variety in $\P^N_k$ such that the Artin--Mazur height of $\var{X}$ is infinite and $\dim \var{X} \geq 2$.
If $p=2$, then there exist $W(k)$-lifts $X_1$ and $X_2$ of $\var{X}$ such that $X_1$ is not lim-perfectoid split and $X_2$ is perfectoid split.
Indeed, let $\bar f$ be an equation defining $\var{X}$.
By \cite{TY26}*{Theorem~5.11}, there exist $W(k)$-lifts of equations $f_1$ and $f_2$ of $f$ such that $\mathrm{ns}(f_1)<\infty$ and $\mathrm{ns}(f_2)=\infty$.
By \cite{yoshikawa2025Criterion}*{Theorem~A and B}, $(A/f_1)^{\wedge p}$ is not perfectoid pure and $(A/f_2)^{\wedge p}$ is perfectoid pure, where
\[
A=W(k)[x_0,\ldots,x_N].
\]
Thus, $X_1=\Proj(A/f_1)$ is not lim-perfectoid split and $X_2=\Proj(A/f_2)$ is perfectoid split by \Cref{Global-local-CY}. 
\end{enumerate}
\end{example}

\begin{example}\label{non-lci}
Let \(p\neq 2\), let \(k\) be a perfect field of characteristic \(p\), and put \(W=W(k)\).  
Consider the Fermat quartic K3 surface
\[
X
=
\operatorname{Proj}
\frac{W[x,y,z,w]}{(x^4+y^4+z^4+w^4)}
\subseteq \mathbf P^3_W.
\]
By \cite{huybrechts2016Lectures}*{Chapter~17,~Proposition~2.10}, the specialization map
\[
\Pic(X_{\Q}) \to \Pic(X_{\F_p})
\]
is injective.
Since the rank of $\Pic(X_{\Q})$ is $20$, there is an ample line bundle $L$ such that $L$ is not isomorphic to $\cO_X(m)$ for every $m \in \Z_{\>0}$.
Let $R$ be a section ring with respect to $L$.

Since \(p\neq 2\), the \(p\)-adic completion of the coordinate ring of $X$ is perfectoid pure by \cite{yoshikawa2025Criterion}*{Example~7.8}.
By \Cref{every-section-ring}, the $p$-adic completion $R^{\wedge p}$ is lim-perfectoid pure. 
\end{example}



\appendix

\section{Dualizing module on affine cone}

\begin{lemma}\label{matlis-dual}
Let $(B,\m)$ be a Noetherian local ring with a dualizing complex $\omega_B^{\bullet}$ and $X$ a proper scheme over $\Spec B$.
Let $\omega_X^{\bullet}\defeq (X \to \Spec B)^{!}\omega_B^{\bullet}$ and $\omega_X$ a dualizing sheaf associated to $\omega_X^{\bullet}$.
For any object \(\mcalF\) of \(\mcalD_{\coh}^b(X, \cO_X)\), we have an isomorphism
\begin{equation*}
    \Hom_B(H^0_{\m}(X, \omega_X^{\bullet} \otimes^L_{\cO_X} \mcalF), E) \cong \comp{\m}{\Hom_{\cO_X}(\mcalF, \cO_X)},
\end{equation*}
where $E$ is an injective hull of $B/\m$.
If \(\mcalF\) is concentrated in degree \(0\), the following isomorphism also holds:
\[
\Hom_B(H^d_{\m}(X,\omega_X \otimes_{\cO_X} \mcalF),E) \cong \comp{\m}{\Hom_{\cO_X}(\mcalF,\cO_X)},
\]
where $d \defeq \dim X$.
\end{lemma}

\begin{proof}
We have 
\begin{align*}
R\Hom_{B}(R\Gamma_{\m}(X,\omega^{\bullet}_X \otimes^L_{\cO_X} \mcalF),E)
&\cong \dcomp{\m}{R\Hom_{B}(R\Gamma(X, \omega^{\bullet}_X \otimes^L_{\cO_X} \mcalF),\omega_{B}^{\bullet})} \\
&\cong \dcomp{\m}{R\Hom_{\cO_X}(\omega_X^{\bullet}\otimes^L_{\cO_X} \mcalF,\omega_X^{\bullet})} \\
&\cong \dcomp{\m}{R\Hom_{\cO_X}(\mcalF,\cO_X)},
\end{align*}
where the first isomorphism follows from local duality (\citeSta{0A84}) and the second is Grothendieck duality for \((X \to \Spec B)^!\) (\citeSta{0AU3}).
Since \(E\) is an injective \(B\)-module, taking the \(0\)-th cohomology yields the isomorphism
\begin{equation*}
    \Hom_B(H^0_{\m}(X, \omega^{\bullet}_X \otimes^L_{\cO_X} \mcalF), E) \cong \comp{\m}{\Hom_{\cO_X}(\mcalF, \cO_X)}
\end{equation*}
by \citeSta{0A06}.

Next, assume that \(\mcalF\) is concentrated in degree \(0\), i.e., \(\mcalF\) is a coherent sheaf on \(X\).

Since \(\mcalF\) is coherent and the dimension of the support of \(\mcalH^{-j}(\omega_X^{\bullet})\) is at most \(j\) for all \(j \geq 0\), the dimension of the support of \(\mcalH^{-(d-i)}(\omega_X^{\bullet} \otimes^L_{\cO_X} \mcalF)\) is strictly less than \(d-i\) for \(i > 0\). By Grothendieck's vanishing theorem for local cohomology, we have
\[
H^{d-i}_{\m}(X,\mcalH^{-(d-i)}(\omega^{\bullet}_X \otimes^L_{\cO_X} \mcalF))=0
\]
for $i >0$.
Using the spectral sequence
\[
E_2^{p,q} = H^p_{\m}(X, \mcalH^q(\omega^{\bullet}_X \otimes^L_{\cO_X} \mcalF)) \Rightarrow H^{p+q}_{\m}(X, \omega^{\bullet}_X \otimes^L_{\cO_X} \mcalF),
\]
the only non-vanishing term contributing to total degree \(0\) is \(E_2^{d,-d}\), which yields
\[
H^0_{\m}(X,\omega^{\bullet}_X \otimes^L_{\cO_X} \mcalF) \cong H^d_{\m}(X,\mcalH^{-d}(\omega^{\bullet}_X \otimes^L_{\cO_X} \mcalF)) \cong H^d_{\m}(X,\omega_X \otimes_{\cO_X} \mcalF),
\]
\end{proof}

\begin{notation}\label{notation-dualizing}
Let $(R_0,\m_0)$ be a Noetherian local ring with a normalized dualizing complex $\omega_{R_0}^{\bullet}$, and let $X$ be a normal projective scheme over $R_0$ such that $H^0(X,\cO_X) \cong R_0$. 
Set $d\defeq \dim X$, $\omega_X^{\bullet} \defeq (X \to \Spec R_0)^{!}\omega_{R_0}^{\bullet}$, and $\omega_X\defeq \mcalH^{-d}(\omega_X^{\bullet})$.
We note that $\omega_X$ is a dualizing sheaf since $\omega_X^{\bullet}$ is a dualizing complex by \cite{KTTWYY1}*{Proposition~2.5}.
We assume that $d > \dim R_0$.
Let $L$ be an ample line bundle on $X$.
Set $R\defeq \bigoplus_{m \geq 0} H^0(X,L^m)$.
\end{notation}

\begin{proposition}\label{dualizing-module-affine-cone}
Assume the setup of \Cref{notation-dualizing}.
Then
\[
\bigoplus_{m \in \Z} H^0(X,\omega_X \otimes_{\cO_X} L^m)
\]
is a dualizing module on $R$.
\end{proposition}

\begin{proof}
We note that the natural map $\Spec R \to \Spec R_0$ is of finite type by the ampleness of $L$.
By the fiber sequence
\[
R\Gamma_{R_+}(R) \to R[0] \to \bigoplus_{m \in \Z} R\Gamma(X,L^m),
\]
we have an isomorphism
\[
\bigoplus_{m \in \Z} H^i_{\m_0}(X,L^m) \xrightarrow{\simeq} H^{i+1}_{\m}(R)
\]
for $i > \dim R_0$, where $\m\defeq (\m_0,R_+)$. Since $\dim X = d$, Grothendieck's vanishing and non-vanishing theorems imply that
\[
H^{d+1}_{\m}(R) \neq 0, \quad \text{and} \quad H^{d+2}_{\m}(R) = 0.
\]
Therefore, we have $\dim R_{\m}=d+1$.
Since $R$ is catenary, we have
\[
\dim R_0+\Ht(R_+)=\dim R_{\m}=d+1,
\]
thus by the assumption $d > \dim R_0$, we obtain
\begin{equation}\label{eq:dim}
    \Ht(R_+) = d+1-\dim R_0 \geq 2.
\end{equation}

We set
\[
\omega_R^{\bullet}\defeq (\Spec R \to \Spec R_0)^{!}(\omega_{R_0}^{\bullet})
\]
and
\[
\pi \colon U\defeq \Spec R \backslash V(R_+) \to X,
\]
where $R_+\defeq \bigoplus_{m \in \Z_{\geq 1}} H^0(X,L^m)$.
Since $\pi$ is affine and
\[
\pi_*\cO_U=\bigoplus_{m \in \Z} L^m,
\]
the morphism $\pi$ is smooth and $\Omega_{U/X} \simeq \cO_U$; in particular, we have $\pi^{!}\simeq \pi^*[1]$ by \citeSta{0AU3}~(9).
Thus, letting $\wt{\omega_R}$ be the sheafification of $\omega_R$, we have
Thus, we have
\[
\wt{\omega_R}|_{U} \simeq \pi^*\omega_X. 
\]
Since $\omega_R$ satisfies the $(S_2)$-condition and $U$ contains all points of codimension at most one in $\Spec R$ by \eqref{eq:dim}, we have
\[
\omega_R \simeq H^0(U,\wt{\omega_R}|_{U}) \simeq H^0(X,\pi_*\pi^*\omega_X) \simeq \bigoplus_{m \in \Z} H^0(X,\omega_X \otimes_{\cO_X} L^m),
\]
as desired.
\end{proof}

\begin{proposition}\label{socle}
Assume the setup of \Cref{notation-dualizing}.
Let $\omega_R$ be the dualizing module constructed in \Cref{dualizing-module-affine-cone}.
Then the fiber sequence
\begin{equation}\label{eq:fiber-seq}
    R\Gamma_{\m}(\omega_R) \to R\Gamma_{\m_0}(\omega_R) \to \bigoplus_{m \in \Z}R\Gamma_{\m_0}(X,\omega_X \otimes_{\cO_X} L^m)  
\end{equation}
induces an injection
\[
H^d_{\m_0}(X,\omega_X) \hookrightarrow H^{d+1}_\m(\omega_R)
\]
and the image of this homomorphism contains the socle of $H^{d+1}_\m(\omega_R)$.    
\end{proposition}

\begin{proof}
By \eqref{eq:fiber-seq} and $\dim R_0<d$, Grothendieck's vanishing theorem implies
\[
H^d_{\m_0}(\omega_R) = 0 \quad \text{and} \quad H^{d+1}_{\m_0}(\omega_R)=0.
\]
Thus, we have an isomorphism
\[
\bigoplus_{m \in \Z} H^d_{\m_0}(X,\omega_X\otimes_{\cO_X} L^m) \xrightarrow{\simeq} H^{d+1}_\m(\omega_R).
\]
Thus, we obtain the injection
\[
\varphi \colon H^d_{\m_0}(X,\omega_X) \hookrightarrow \bigoplus_{m \in \Z} H^d_{\m_0}(X,\omega_X\otimes_{\cO_X} L^m) \xrightarrow{\simeq}  H^{d+1}_\m(\omega_R).
\]
Applying \Cref{matlis-dual} to $X=\Spec B=\Spec R$, we have
\[
\Hom_{R}(R/\m,H^{d+1}_\m(\omega_R)) \simeq \Hom_{R}(R/\m,E_R) \simeq R/\m,
\]
where $E_R$ is the injective hull of $R/\m$.
Thus, the socle of $H^{d+1}_\m(\omega_R)$ is a one-dimensional vector space over $R/\m$.
Note that the action of $R_+$ on \(\bigoplus_{m \in \Z} H^d_{\m_0}(X,\omega_X\otimes_{\cO_X} L^m)\) strictly increases the degree. Therefore, to show that the socle is contained in degree \(0\), it is enough to show that
\[
H^d_{\m_0}(X,\omega_X) \neq 0 \quad \text{and} \quad H^d_{\m_0}(X,\omega_X \otimes_{\cO_X} L^m)=0
\]
for every $m \geq 1$.
Applying \Cref{matlis-dual} to $X \to \Spec R_0$, we obtain
\[
\Hom_{R_0}(H^d_{\m_0}(X, \omega_X \otimes_{\cO_X} L^m), E_{R_0}) \simeq \comp{\m_0}{H^0(X, L^{-m})},
\]
where \(E_{R_0}\) is the injective hull of \(R_0/\m_0\).
Since \(\mcalL\) is ample and \(X\) is normal projective, we have \(H^0(X, \mcalL^{-m})=0\) for \(m \geq 1\), while
\[
\comp{\m_0}{H^0(X,\cO_X)} \simeq \comp{\m_0}{R_0} \neq 0.
\]
This proves the claim.
\end{proof}

\begin{proposition} \label{CMCriterion}
We use the setting of Notation~\ref{SettingFibersequenceRecall}, and let \(\mathfrak m=(\mathfrak m_0,R_+)\).
Then the following are equivalent:
\begin{enumerate}
    \item \(R_{\mathfrak m}\) is Cohen--Macaulay.
    \item 
    \[
    H^i_{(p)}(X,\mathcal O_X(m))=0 \qquad (2\le i\le d-1,\ m\in \mathbb Z),
    \]
    and \(H^1(X,\mathcal O_X(m))\) is \(p\)-torsion free for all \(m\in \mathbb Z\).
\end{enumerate}
\end{proposition}

\begin{proof}
Consider the exact triangle
\[
R\Gamma_{\mathfrak m}(R)\to R\Gamma_{(p)}(R)\to \bigoplus_{m\in\mathbb Z} R\Gamma_{(p)}(X,\mathcal O_X(m)).
\]
Taking the degree-\(n\) part, we obtain an exact triangle
\[
R\Gamma_{\mathfrak m}(R)_n \to R\Gamma_{(p)}(R_n)\to R\Gamma_{(p)}(X,\mathcal O_X(n)).
\]

Since \(R_n\) is \(p\)-torsion free, we have
\[
R\Gamma_{(p)}(R_n)\simeq [\,R_n\to R_n[1/p]\,],
\]
and hence
\[
H^i_{(p)}(R_n)=0 \qquad (i\neq 1), \qquad
H^1_{(p)}(R_n)\cong R_n[1/p]/R_n.
\]
Therefore, noting that \(H^0_{(p)}(X,\mathcal O_X(n))=0\), the associated long exact sequence yields
\[
H^0_{\mathfrak m}(R)_n=H^1_{\mathfrak m}(R)_n=0,
\]
an exact sequence
\[
0\to R_n[1/p]/R_n \to H^1_{(p)}(X,\mathcal O_X(n))
   \to H^2_{\mathfrak m}(R)_n \to 0,
\]
and isomorphisms
\[
H^i_{\mathfrak m}(R)_n \cong H^{i-1}_{(p)}(X,\mathcal O_X(n))
\qquad (i\ge 3).
\]

Now consider the exact triangle
\[
R\Gamma(X,\mathcal O_X(n)) \to R\Gamma(X,\mathcal O_X(n))[1/p]
\to R\Gamma_{(p)}(X,\mathcal O_X(n))[1].
\]
Since \(H^0(X,\mathcal O_X(n))=R_n\), the associated long exact sequence yields
\[
0\to R_n[1/p]/R_n \to H^1_{(p)}(X,\mathcal O_X(n))
   \to H^1(X,\mathcal O_X(n)) \to
   H^1(X,\mathcal O_X(n))[1/p].
\]
Hence
\[
0\to R_n[1/p]/R_n \to H^1_{(p)}(X,\mathcal O_X(n))
   \to H^1(X,\mathcal O_X(n))[p^\infty]\to 0.
\]
Combining this with the previous exact sequence, we obtain a natural isomorphism
\[
H^2_{\mathfrak m}(R)_n \cong H^1(X,\mathcal O_X(n))[p^\infty].
\]
Therefore, \(H^2_{\mathfrak m}(R)=0\) if and only if \(H^1(X,\mathcal O_X(n))\) is \(p\)-torsion free for all \(n\in\mathbb Z\).

For \(i\ge 3\), the isomorphism
\[
H^i_{\mathfrak m}(R)_n \cong H^{i-1}_{(p)}(X,\mathcal O_X(n))
\]
shows that \(H^i_{\mathfrak m}(R)=0\) if and only if \(H^{i-1}_{(p)}(X,\mathcal O_X(n))=0\) for all \(n\in\mathbb Z\).

Since we have already seen that \(H^i_{\mathfrak m}(R)=0\) for \(i=0, 1\), the condition that \(H^i_{\mathfrak m}(R)=0\) for all \(0\le i\le d\) is equivalent to the conjunction of the following two conditions:
\begin{enumerate}
    \item \(H^j_{(p)}(X,\mathcal O_X(n))=0\) for all \(2\le j\le d-1\) and \(n\in\mathbb Z\);
    \item \(H^1(X,\mathcal O_X(n))\) is \(p\)-torsion free for all \(n\in\mathbb Z\).
\end{enumerate}

Finally, since \(\dim R_{\mathfrak m}=d+1\), the ring \(R_{\mathfrak m}\) is Cohen--Macaulay if and only if
\[
H^i_{\mathfrak m}(R)=0 \qquad (i<d+1).
\]
This proves the assertion.
\end{proof}

\begin{proposition} \label{CYSectionRingCM}
Keep the setting of \Cref{SettingFibersequenceRecall}.
Assume that 
\[
H^i_{(p)}(X,\cO_X(m))=0 \quad \text{for $m<0$ and $i<\dim X$}.
\]
    Assume further one of the following conditions:
    \begin{enumerate}
        \item \(\omega_X\) is trivial and \(H^i(X, \mcalO_X) = 0\) for \(0 < i < d-1\);
        \item \(\omega_X\) is anti-ample and \(\omega_X^{-1} = L\).
    \end{enumerate}
    Then the localization \(R_{\mfrakm}\) of the section ring \(R = \bigoplus_{n \geq 0} H^0(X, \mcalO_X(n))\) is Cohen--Macaulay.
\end{proposition}

\begin{proof}
    By \Cref{CMCriterion}, it suffices to show that \(H^i_{(p)}(X, \mcalO_X(m)) = 0\) for \(m \in \setZ\) and \(2 \leq i \leq d-1\) and that \(H^1(X, \mcalO_X(m))\) is \(p\)-torsion-free for all \(m \in \setZ\).

    First, we show that \(H^i_{(p)}(X, \mcalO_X(m)) = 0\) for any \(m \geq 0\) and \(1 \leq i \leq  d-1\).
    By assumption, we have \(H^i_{(p)}(X, \mcalO_X(m)) = 0\) for \(m < 0\) and \(i \leq d-1\).

    (1): Assume that \(\omega_X\) is trivial and \(H^i(X, \mcalO_X) = 0\) for \(0 < i < d-1\).
    Using the duality in \Cref{matlis-dual}, we have an isomorphism
    \begin{equation*}
        H^i(X, \mcalO_X(m)) \cong H^{d-i}_{(p)}(X, \mcalO_X(-m))^{\vee},
    \end{equation*}
    which implies \(H^i(X, \mcalO_X(m)) = 0\) for \(m > 0\) and \(i \geq 1\).
    In particular, we have
    \begin{equation} \label{RGammapVanishing}
        R\Gamma_{(p)}(X, \mcalO_X(m)) \cong R\Gamma_{(p)}(H^0(X, \mcalO_X(m))) \in \mcalD^{[0, 1]}(R_0)
    \end{equation}
    for \(m > 0\) and thus \(H^i_{(p)}(X, \mcalO_X(m)) = 0\) holds for \(m > 0\) and \(i \geq 2\).
    Using the duality again, the assumption \(H^i(X, \mcalO_X) = 0\) for \(1 \leq i \leq d-2\) implies that \(H^i_{(p)}(X, \mcalO_X) = 0\) for \(2 \leq i \leq d-1\).

    (2): Assume that \(\omega_X\) is anti-ample and \(\omega_X^{-1} = L\).
    Using the duality in \Cref{matlis-dual}, we have an isomorphism
    \begin{equation*}
        H_{(p)}^{d-i}(X, \mcalO_X(m-1))^{\vee} \cong H^i(X, \mcalO_X(-m)),
    \end{equation*}
    which implies \(H^i(X, \mcalO_X(m)) = 0\) for \(m \geq 0\) and \(i \geq 1\).
    As in (1), the same isomorphism \eqref{RGammapVanishing} holds for \(m \geq 0\) and thus \(H^i_{(p)}(X, \mcalO_X(m)) = 0\) for \(m \geq 0\) and \(i \geq 2\).

    Finally, we show that \(H^1(X, \mcalO_X(m))\) is \(p\)-torsion-free for all \(m \in \setZ\) in both cases.
    By the above vanishing \(H^1(X, \mcalO_X(m)) = 0\) for \(m \geq 0\), it suffices to show this for \(m < 0\).
    Since $\mcalO_{X_{p=0}}(m)$ is anti-ample, we have $H^0(X_{p=0}, \mcalO_{X_{p=0}}(m))=0$.
    Thus, it follows from the exact sequence
    \[
0 \to \cO_X(m) \xrightarrow{\cdot p} \cO_X(m) \to \cO_{X_{p=0}}(m) \to 0
    \]
    that \(H^1(X, \mcalO_X(m))\) is \(p\)-torsion-free, as desired.
\end{proof}


\end{document}